\documentclass[a4paper,reqno]{amsart}
\usepackage{amscd,amsfonts,amssymb}
\usepackage{textgreek}
\usepackage{tikz}
\usepackage{pgfplots}
\usetikzlibrary{automata,positioning,calc,trees}
\usetikzlibrary{intersections,pgfplots.fillbetween}
\usepackage{graphicx,tikz}
\usepackage{pgf,tikz,pgfplots}
\usepackage{mathrsfs}
\usetikzlibrary{arrows}
\usepackage{enumerate}
\usepackage[shortlabels]{enumitem}
\usepackage{mathrsfs}
\usepackage{amssymb,amsmath,amsthm,color}
\usepackage{caption,subcaption}
\usepackage{hyperref}
\usepackage{cleveref}
\usepackage[initials,nobysame,bibtex-style]{amsrefs}
\usepackage{url}
\usepackage{setspace}
\usepackage{float}
\usepackage{makecell}
\tikzstyle{startstop} = [rectangle, rounded corners, minimum width=3cm, minimum height=1cm,text centered, draw=black]
\tikzstyle{process} = [rectangle, rounded corners, minimum width=3cm, minimum height=1cm, text centered, draw=black]
\tikzstyle{arrow} = [thick,->,>=stealth]

\BibSpec{article}{%
	+{}  {\PrintAuthors}                {author}
	+{,} { \textit}                     {title}
	+{.} { }                            {part}
	+{:} { \textit}                     {subtitle}
	+{,} { \PrintContributions}         {contribution}
	+{.} { \PrintPartials}              {partial}
	+{,} { }                            {journal}
	+{}  { \textbf}                     {volume}
	+{}  { \PrintDatePV}                {date}
	+{,} { \issuetext}                  {number}
	+{,} { \eprintpages}                {pages}
	+{,} { }                            {status}
	+{,} { \PrintDOI}                   {doi}
	+{,} { \eprint}        				{eprint}
	+{}  { \parenthesize}               {language}
	+{}  { \PrintTranslation}           {translation}
	+{;} { \PrintReprint}               {reprint}
	+{.} { }                            {note}
	+{.} {}                             {transition}
}
\date{\today}
\allowdisplaybreaks
\newcommand{\R}{{\mathbb R}}       
\newcommand{\N}{{\mathbb N}}
\newcommand{\Sp}{{\mathbb S}}
\newcommand{\Z}{{\mathbb Z}}    
\newcommand{\supp}[1]{\operatorname{supp}(#1)}   

\newtheorem{theorem}{Theorem}[section]
\newtheorem{lemma}[theorem]{Lemma}

\newtheorem{proposition}[theorem]{Proposition}

\newtheorem*{theorem*}{Theorem}

\theoremstyle{definition}

\newtheorem{brem}[theorem]{Remark}

\numberwithin{equation}{section}

\begin{document}
	
\title[Bilinear rough singular integrals]{An alternate approach to bilinear rough singular integrals}
\author[Bhojak]{Ankit Bhojak}
	\address{Ankit Bhojak\\
		Department of Mathematics\\
		Indian Institute of Science Education and Research Bhopal\\
		Bhopal-462066, India.}
	\email{ankitb@iiserb.ac.in, ankitbhojak101@gmail.com}

\author[Shrivastava]{Saurabh Shrivastava}
	\address{Saurabh Shrivastava\\
		Department of Mathematics\\
		Indian Institute of Science Education and Research Bhopal\\
		Bhopal-462066, India.}
	\email{saurabhk@iiserb.ac.in}
\thanks{}
\begin{abstract}
	The goal of this paper is to provide a new approach to the $L^p-$boundedness of bilinear rough singular integral operators. This approach relies on local Fourier series expansion of input functions leading to trilinear estimates with desired decay in the frequency parameter. This approach differs from the existing methods based on wavelet decomposition of the multiplier employed in the work of Grafakos, He and Honz\'ik and in a series of subsequent papers in the context of bilinear rough singular integrals. Using this new approach, we establish sharp $L^p-$estimates for maximally truncated bilinear rough singular integrals when the kernel is supported away from the diagonal in the plane.  Furthermore, this method provides a novel, self-contained proof of $L^p-$boundedness of the bilinear rough singular integral operators across all dimensions for the optimal range of exponents. 
\end{abstract}

\maketitle
	
\section{Introduction}
The study of singular integrals of the form $Tf(x)=p.v.\int_{\R^d} K(y)f(x-y)dy$ is a central theme in harmonic analysis. For kernels of the form  $K(y)=\Omega(\frac{y}{|y|})|y|^{-d},~y\neq 0,$ where $\Omega\in L\log L(\Sp^{d-1})$ satisfies the mean zero condition, the $L^p-$boundedness of the corresponding operator $T$ for $1<p<\infty$ dates back to the work of Calder\'{o}n and Zygmund~\cite{CZ1956}. An alternative proof of these $L^p-$estimates for $Tf$, based on a double dyadic decomposition of the kernel and the Fourier transform estimates, was later obtained by Duoandikoetxea and Rubio de Francia~\cite{DR1986}. The issue of weak type estimates at the endpoint $p=1$ was addressed by Christ and Rubio de Francia~\cite{Christ1988,CR1988} and Hofmann~\cite{Hofmann1988} independently for dimension two,  and later by Seeger~\cite{Seeger1996} for all dimensions. We refer to \cite{Honzik2020,BM2023,Lai2025,BS2026} for recent progress concerning the endpoint bounds for the corresponding maximal rough singular integral operators.

We investigate the bilinear analogues of rough singular integral operators, which have attracted significant research interest in recent years. We review and discuss several key developments in this direction. Let $\Omega\in L^1(\Sp^{2d-1})$ satisfy the mean zero condition $\int_{\Sp^{2d-1}}\Omega(\theta)d\sigma(\theta)=0$, where $d\sigma$ is the surface measure on $\Sp^{2d-1}$. The bilinear rough singular integral operator $T_{\Omega}$ is defined as
\[T_{\Omega}(f_1,f_2)(x)=\int_{\R^{2d}}\Omega\left(\frac{(y_1,y_2)}{|(y_1,y_2)|}\right){|(y_1,y_2)|^{-2d}}f_1(x-y_1)f_2(x-y_2)\;dy_1dy_2.\]
Coifman and Meyer \cite[Theorem I]{CF1975} showed that if $\Omega$ is a function of bounded variation on $\Sp^1$, then $T_\Omega$ is bounded from $L^{p_1}(\R)\times L^{p_2}(\R)$ to $L^{p}(\R)$ for $(p_1,p_2,p)\in\mathcal{H}^1$, where 
\[\mathcal{H}^q=\left\{(p_1,p_2,p):\;1<p_1,p_2<\infty,\;\frac{1}{2}<p<\infty,\;\frac{1}{p_1}+\frac{1}{p_2}=\frac{1}{p},\;\text{and}\;\frac{1}{p}+\frac{1}{q}<2\right\},\;1\leq q\leq\infty.\]
In higher dimensions, Kenig and Stein \cite{KenigStein} and Grafakos and Torres \cite{GT2002} proved the $L^{p_1}\times L^{p_2}\rightarrow L^p-$boundedness of $T_\Omega$, where $\Omega$ is a Lipschitz function and exponents lie in the larger range $\mathcal{H}^\infty$. 

We note that when $\Omega$ is an odd function on $\Sp^1$, the operator $T_\Omega$ can be expressed as
\[T_\Omega(f_1,f_2)(x)=\int_{\Sp^1}\Omega(\theta_1,\theta_2)BH_{\theta_1,\theta_2}(f_1,f_2)(x)d\sigma(\theta_1,\theta_2),\]
where $BH_{\theta_1,\theta_2}$ is the bilinear Hilbert transform defined by
\[BH_{\theta_1,\theta_2}(f_1,f_2)(x)=\int_\R f_1(x-\theta_1t)f_2(x-\theta_2t)\frac{dt}{t}.\]
The boundedness of $BH_{\theta_1,\theta_2}$ was proved in the seminal works of Lacey and Thiele \cite{LT1997, LT1999}. By exploiting the uniform boundedness of $BH_{\theta_1,\theta_2}$ \cite{GL2004,Li2006}, the authors in \cite{DGHST2011} established the boundedness of $T_\Omega$ for $(p_1,p_2,p)\in\mathcal{H}^1$ satisfying the condition  $\max\left\{\left|\frac{1}{p_1}-\frac{1}{p_2}\right|,\left|\frac{1}{p_2}-\frac{1}{p'}\right|,\left|\frac{1}{p'}-\frac{1}{p_1}\right|\right\}<\frac{1}{2}$ whenever the even part of $\Omega$ lies in the Hardy space $H^1(\Sp^1)$. 

The first complete result for the bilinear operator $T_\Omega$ without any smoothness assumption on $\Omega$ was obtained by Grafakos, He and Honz\'ik in \cite{GHH2018}. They proved that $T_\Omega$ is bounded from $L^{p_1}\times L^{p_2}\rightarrow L^p$ for exponents $(p_1,p_2,p)$ in the range $\mathcal{H}^\infty$ and $\Omega\in L^\infty(\Sp^{2d-1})$. Their breakthrough approach relied on a double dyadic decomposition of the kernel following the work of Duoandikoetxea and Rubio de Francia~\cite{DR1986}, combined with a further decomposition of the underlying multiplier using a tensor-type, compactly supported wavelet basis with vanishing moments. These techniques coupled with combinatorial arguments yielded $L^2\times L^2\to L^1-$estimate with appropriate decay in the frequency parameter whenever $\Omega\in L^2(\Sp^{2d-1})$. Further, since the bilinear Calder\'on-Zygmund theory developed by Grafakos and Torres~\cite{GT2002} behaves well for $\Omega\in L^\infty(\Sp^{2d-1})$, an interpolation argument for bilinear operators applied to the two types of estimates mentioned above yields the sharp $L^{p_1}\times L^{p_2}\rightarrow L^p-$boundedness results for the operator $T_\Omega$ when $\Omega\in L^\infty(\Sp^{2d-1})$. 

Grafakos, He and Honz\'ik \cite{GHH2018} also established the $L^{p_1}\times L^{p_2}\rightarrow L^p-$boundedness of $T_\Omega$ in the local $L^2$ range: $2<p_1,p_2<\infty,\;1\leq p\leq 2$ for $\Omega$ belonging to a larger space $L^2(\Sp^{2d-1})$. Subsequently, Grafakos, He and Slavikova~\cite{GHS2020} extended this result for $\Omega\in L^q(\Sp^{2d-1})$ with $q>\frac{4}{3}$. Building upon the $(2,2,1)-$inequality from \cite{GHH2018}, He and Park \cite{HP2023} improved the range of boundedness of $T_\Omega$ to exponents in $\mathcal{H}^q$ for $\Omega\in L^q(\Sp^{2d-1}),\;q>\frac{4}{3}$. In doing so, they developed bilinear Calder\'on-Zygmund theory suitable for all $1<q\leq\infty$. We refer the reader to \cite{GHHP2023,GHHP2024} for multilinear extensions of this result. In \cite{GHS2019}, the authors showed that $\mathcal{H}^q$ is the largest open range for which the operator $T_\Omega$ is bounded for $\Omega\in L^q,\;1<q\leq\infty$. Finally, Dosidis and Slav\'ikov\'a \cite{DosidisSlavikova} showed that $T_\Omega$ is also bounded in the range $\mathcal{H}^q$ whenever $\Omega\in L^q(\Sp^{2d-1}),\;1<q\leq\infty$.  In this work, they utilized the boundedness of shifted square functions, whose details may be found in \cite{Muscalu}, to obtain admissible growth estimates for Littlewood-Paley pieces of the operator $T_\Omega$ in the Banach range: $1\leq p_1,p_2,p\leq \infty$ and  $\Omega\in L^1(\Sp^{2d-1})$. They also showed that the strong type boundedness of $T_\Omega$ cannot hold outside the region $\mathcal{H}^q$ for $1\leq p_1,p_2<\infty$. However, the question of weak type boundedness of $T_\Omega$ for exponents on the lines $p_1=1$, $p_2=1$ and $\frac{1}{p}+\frac{1}{q}=2$ remains open.
\section{Main results and methodology}
In this paper we provide an alternate approach to the boundedness of the operator $T_\Omega$ for the complete range of exponents $\mathcal{H}^q$ for $\Omega\in L^q(\Sp^{2d-1}),\;1<q\leq\infty$ across all dimensions. In particular, the main result of the paper in this direction is the following theorem.
\begin{theorem}\label{Thm:Rough}
	Let $d\geq1$, $1<q\leq\infty$ and $(p_1,p_2,p)\in\mathcal{H}^q$. Then, for all $\Omega\in L^q(\Sp^{2d-1})$ with $\int_{\Sp^{2d-1}}\Omega(\theta) d\sigma(\theta)=0$, we have
	\[\|T_\Omega(f_1,f_2)\|_{L^p(\R^d)}\lesssim\|\Omega\|_{L^q(\Sp^{2d-1})}\|f_1\|_{L^{p_1}(\R^d)}\|f_2\|_{L^{p_2}(\R^d)}.\]
\end{theorem}
In the expression above, the notation $A\lesssim B$ means that there exists a constant $C>0$ (independent of $A$ and $B$) such that $A\leq CB$.

Our proof of \Cref{Thm:Rough} involves a standard spatial localization of the kernel to obtain a localized kernel of the form
\[K^0(y_1,y_2)=\Omega\left(\frac{(y_1,y_2)}{|(y_1,y_2)|}\right){|(y_1,y_2)|^{-2d}}\beta(|(y_1,y_2)|),\]
where $\beta\in C_c^\infty(\R)$ with $\supp{\beta}\subset(2^{-1},2)$. Let $T^0_\Omega$ be the local operator corresponding to $K^0$ defined by
\[T^0_\Omega(f_1,f_2)(x)=\int_{\R^{2d}}K^0(y_1,y_2)f_1(x-y_1)f_2(x-y_2)\;dy_1dy_2.\]
We prove a frequency-decay trilinear estimate for the operator $T_{\Omega }^{0}$ acting on functions which are Fourier localized on an annulus. Specifically, we establish the following theorem.
\begin{theorem}\label{thm:Smoothing}
	Let $d\geq1$ and $1<q\leq\infty$. Then, there exists $c>0$ such that for all $\lambda\geq1$, we have
		\begin{equation}\label{inq:Trilinearsmoothing1}
			|\langle T^0_\Omega(f_1,f_2),f_3\rangle|\lesssim\lambda^{-c}\|\Omega\|_{L^q(\Sp^{2d-1})}\|f_{l_1}\|_{L^{2}(\R^d)}\|f_{l_2}\|_{L^{2}(\R^d)}\|f_{l_3}\|_{L^{\infty}(\R^d)},
		\end{equation}
	whenever $\supp{\widehat{f_{l_1}}}\cup\supp{\widehat{f_{l_2}}}\subset\left\{\lambda\leq|\xi|\leq2\lambda\right\}$ and $\supp{\widehat{f_{l_3}}}\subset\left\{|\xi|\leq2\lambda\right\}$ for $\{l_1,l_2,l_3\}=\{1,2,3\}$.
\end{theorem}
Christ and Zhou~\cite{ChristZhou} obtained the trilinear smoothing estimates in the context of bilinear maximal functions associated to a certain class of curves in the plane. Our proof of \Cref{thm:Smoothing} adapts their core framework involving the local Fourier series expansion of the test functions, but introduces novel key modifications specifically in the higher dimensional situation. In principle, the existing method of  wavelet decomposition employed in~\cite{GHH2018} is replaced by the local Fourier series expansion of functions. Further, we use the polar decomposition of the local bilinear operator $T_\Omega^0$ and perform a careful oscillatory phase analysis with respect to the spatial and radial variables in the polar decomposition of the operator $T_\Omega^0$ leading to the trilinear estimate \eqref{inq:Trilinearsmoothing1}. In dimension one ($d=1$), the proof of \eqref{inq:Trilinearsmoothing1} is relatively simpler. This is a consequence of the fact that the oscillatory decay is enough to accommodate the loss from summing the Fourier coefficients of one-dimensional input functions via a straightforward application of Cauchy-Schwarz inequality. A major departure from the one-dimensional case is our use of an intricate counting argument based on an interpolation scheme between the \(\ell ^{2}\) and \(\ell ^{\infty}\) estimates of the oscillatory integral coefficients to bound the product of the Fourier coefficients of the higher dimensional input functions. We carry out the proof of \Cref{thm:Smoothing} in \Cref{sec:proofsmoothing}. 

The decay estimate from \Cref{thm:Smoothing} plays the key role in the proof of \Cref{Thm:Rough}. Next, we obtain admissible growth estimates of the operator in the Banach and non-Banach ranges. In the Banach range, the proof  relies on the estimates for shifted square functions. We refer to \cite{DosidisSlavikova} for more details about the shifted square functions. The admissible growth estimates for exponents in the non-Banach range are based on single scale estimates obtained in \Cref{sec:singlescale} and the bilinear Calder\'on-Zygmund theory adapted to the boundary $\frac{1}{p}+\frac{1}{q}=2$. We refer the reader to \cite{ChristZhou,HP2023} for a similar approach. The proof of \Cref{Thm:Rough} is given in \Cref{sec:proofrough}.

Next, we extend the $L^p-$boundedness of $T_\Omega$ for $\Omega$ belonging to a larger class of Orlicz spaces.  In particular, we prove the following theorem. 
\begin{theorem}\label{Thm:RoughBanach}
	Let $d\geq1$, $1\leq p_1,p_2,p<\infty$ with $\frac{1}{p_1}+\frac{1}{p_2}=\frac{1}{p}$. Let $A=A(p_1,p_2,p)$ be defined as
	\[A=1+\max\left\{\frac{1}{p_1},\frac{1}{p_2},\frac{1}{p'}\right\}.\]
	Then, for all $\Omega\in L(\log L)^A(\Sp^{2d-1})$ with $\int_{\Sp^{2d-1}}\Omega(\theta) d\sigma(\theta)=0$, we have
	\[\|T_\Omega(f_1,f_2)\|_{L^{p}(\R^d)}\lesssim\|\Omega\|_{L(LogL)^A(\Sp^{2d-1})}\|f_1\|_{L^{p_1}(\R^d)}\|f_2\|_{L^{p_2}(\R^d)}.\]
\end{theorem}
Observe that \Cref{Thm:RoughBanach} recovers the result proved in \cite{DPS2024}. The proof of this theorem is given in \Cref{sec:proofrough1}.

Recently, Honz\'ik, Lappas and Slav\'ikov\'a \cite{HLS2025} showed that the $L^{p_1}\times L^{p_2}\rightarrow L^p-$boundedness of the operator $T_\Omega$  extends to a much larger range of exponents $(p_1,p_2,p)$ for all $\Omega\in L^q(\Sp^1)$ under the condition that $\Omega$ is supported away from the diagonal $\{(\theta_1,\theta_2)\in\Sp^1:\theta_1=\theta_2\}$. We present an alternative proof of this result using the techniques from this paper. Specifically, we establish the following theorem. 
\begin{theorem}\label{Thm:Roughantidiag}
	Let $1<q\leq\infty$ and $(p_1,p_2,p)\in\mathcal{H}^\infty$. Then, for all $\Omega\in L^q(\Sp^1)$ with mean zero condition $\int_{\Sp^{1}}\Omega(\theta) d\sigma(\theta)=0$ and $\supp{\Omega}\subset \{(\theta_1,\theta_2)\in\Sp^1:|\theta_1-\theta_2|\geq\frac{1}{2}\}$, we have  
	\[\|T_\Omega(f_1,f_2)\|_{L^p(\R)}\lesssim\|\Omega\|_{L^q(\Sp^{1})}\|f_1\|_{L^{p_1}(\R)}\|f_2\|_{L^{p_2}(\R)}.\]
\end{theorem}
Indeed, the current methodology is so streamlined that \Cref{Thm:Roughantidiag} follows with only minor modifications to the proof of \Cref{Thm:Rough}. For this case, we employ the single-scale estimate \eqref{singlescale:Lpbound_d=1} instead of the estimate \eqref{singlescale:Lpbound}. The estimate  \eqref{singlescale:Lpbound_d=1} is adapted to incorporate the underlying support condition on \(\Omega \) specified in \Cref{Thm:Roughantidiag}.
 
Finally, we show that the methods of this paper are powerful enough to obtain $L^p-$boundedness of the maximally truncated bilinear operator $T_\Omega^*$ defined as follows:
\[T_{\Omega}^*(f_1,f_2)(x)=\sup_{\epsilon>0}\left|\int_{|(y_1,y_2)|>\epsilon}\Omega\left(\frac{(y_1,y_2)}{|(y_1,y_2)|}\right){|(y_1,y_2)|^{-2d}}f_1(x-y_1)f_2(x-y_2)\;dy_1dy_2\right|.\]
The study of $L^p-$boundedness of the maximal operator $T_\Omega^*$ was initiated in \cite{BH2019}. They exploited the method of wavelet decomposition of the multiplier from \cite{GHH2018} to show that the operator $T_{\Omega}^*$ is bounded in the domain $\mathcal{H}^\infty$ whenever $\Omega\in L^\infty(\Sp^{2d-1})$. This result was further improved in \cite{GHHP2024ae} and the optimal $L^p-$boundedness of the operator in the range  $\mathcal{H}^q$ for $\Omega\in L^q(\Sp^{2d-1})$ was obtained in \cite{Park2025}. We demonstrate how our methods yield $L^p-$bounds for the rough maximal singular integral in the optimal range $\mathcal{H}^q$. More precisely, we establish the following result. 
\begin{theorem}\label{Thm:Roughmaximal}
	Let $d\geq1$, $1<q\leq\infty$ and $(p_1,p_2,p)\in\mathcal{H}^q$. Then, for all $\Omega\in L^q(\Sp^{2d-1})$ with $\int_{\Sp^{2d-1}}\Omega(\theta)d\sigma(\theta)=0$, we have
	\[\|T_\Omega^*(f_1,f_2)\|_{L^p(\R^d)}\lesssim\|\Omega\|_{L^q(\Sp^{2d-1})}\|f_1\|_{L^{p_1}(\R^d)}\|f_2\|_{L^{p_2}(\R^d)}.\]
	Moreover, if $\Omega\in L^q(\Sp^{1})$ with mean zero condition and $\supp{\Omega}\subset \{(\theta_1,\theta_2)\in\Sp^1:|\theta_1-\theta_2|\geq\frac{1}{2}\}$, then we have
	\[\|T_\Omega^*(f_1,f_2)\|_{L^p(\R)}\lesssim\|\Omega\|_{L^q(\Sp^{1})}\|f_1\|_{L^{p_1}(\R)}\|f_2\|_{L^{p_2}(\R)},\]
	for $(p_1,p_2,p)\in\mathcal{H}^\infty$.
\end{theorem}
\begin{brem}
We note that \Cref{Thm:Roughmaximal} recovers the result of \cite{Park2025} in all dimensions. Furthermore, the result is new when $\Omega$ is supported outside a neighborhood of the diagonal in the plane. This extends the $L^p-$boundedness result of \Cref{Thm:Roughantidiag} for the maximal operator $T_\Omega^*$.
\end{brem}
In order to prove \Cref{Thm:Roughmaximal}, we need to consider the maximal operator $M_\Omega$ defined as
\[M_\Omega(f_1,f_2)(x)=\sup_{R>0}\frac{1}{R^{2d}}\int_{|(y_1,y_2)|\leq R}\left|\Omega\left(\frac{(y_1,y_2)}{|(y_1,y_2)|}\right)f_1(x-y_1)f_2(x-y_2)\right|dy_1dy_2.\]

The boundedness of $M_\Omega$ in bilinear and multilinear cases was proved in \cite{BH2019} and \cite{GHHP2024ae}, respectively. Their proofs relied on two types of estimates and interpolation between them. The first is an $L^1\times L^1\to L^{\frac{1}{2},\infty}-$estimate under the assumption that $\Omega\in L^\infty(\Sp^{2d-1})$. This estimate is obtained by dominating the operator $M_\Omega$ with the bilinear Hardy-Littlewood maximal function. The second is an estimate for the Banach range of exponents, obtained by dominating $M_\Omega$ with a product of certain directional maximal functions. However, these methods do not allow one to incorporate the support condition on $\Omega$ in \Cref{Thm:Roughantidiag}. In order to exploit this additional fact that $\Omega$ is supported outside a neighborhood of the diagonal $\{\theta_1=\theta_2\}$ in the plane, we require new estimates. We adapt the proof strategy of \Cref{Thm:Rough} to $M_\Omega$ and obtain the following result.
\begin{theorem}\label{Thm:Roughsinglemaximal}
	Let $d\geq1$, $1<q\leq\infty$ and $(p_1,p_2,p)\in\mathcal{H}^q$. Then, for all $\Omega\in L^q(\Sp^{2d-1})$, we have
	\[\|M_\Omega(f_1,f_2)\|_{L^p(\R^d)}\lesssim\|\Omega\|_{L^q(\Sp^{2d-1})}\|f_1\|_{L^{p_1}(\R^d)}\|f_2\|_{L^{p_2}(\R^d)}.\]
	Moreover, if $\Omega\in L^q(\Sp^{1})$ with $\supp{\Omega}\subset \{(\theta_1,\theta_2)\in\Sp^1:|\theta_1-\theta_2|\geq\frac{1}{2}\}$, then we have
	\[\|M_\Omega(f_1,f_2)\|_{L^p(\R)}\lesssim\|\Omega\|_{L^q(\Sp^{1})}\|f_1\|_{L^{p_1}(\R)}\|f_2\|_{L^{p_2}(\R)},\]
	for $(p_1,p_2,p)\in\mathcal{H}^\infty$.
\end{theorem}
We note that the first part of the theorem above, which requires no additional assumption on the support of $\Omega$, recovers the result proved in \cite{BH2019}. The second part of the theorem, where the exponents lie in the improved range $\mathcal{H}^\infty$, is new. The detailed proofs of \Cref{Thm:Roughmaximal} and \Cref{Thm:Roughsinglemaximal} are provided in \Cref{sec:proofmaximal} and \Cref{sec:proofsinglemaximal}, respectively.
\section{Notation and toolbox for bilinear operators}
For any multi-index $\alpha=(\alpha_1,\dots,\alpha_d)$ and $x=(x_1,\dots,x_d)\in\R^d$, we denote
\[|\alpha|_1=\alpha_1+\dots+\alpha_d,\quad\text{and}\quad x^\alpha=x_1^{\alpha_1}\dots x_d^{\alpha_d}.\]
The following result is a bilinear version of the Marcinkiewicz interpolation theorem which appeared in \cite[Lemma 2.1]{HP2023}.
\begin{theorem}[\cite{HP2023}]\label{realinterpolation}
	Let $0<p_{1,k},p_{2,k},p_{k} \leq \infty$ with $\frac{1}{p_{k}}=\frac{1}{p_{1,k}}+\frac{1}{p_{2,k}}$ for $k=1,2,3$. Suppose that $T$ is a bilinear operator having the mapping properties
	\[\left\|T\left(f_1, f_2\right)\right\|_{L^{p_{k}, \infty\left(\mathbb{R}^d\right)}} \leq M_k\left\|f_1\right\|_{L^{p_{1,k}}\left(\mathbb{R}^d\right)}\left\|f_2\right\|_{L^{p_{2,k}}\left(\mathbb{R}^d\right)}, \quad k=1,2,3,\]
	for Schwartz functions $f_1, f_2$ on $\mathbb{R}^d$. Then for any $0<\theta_k<1,\;k=1,2,3$ with $\theta_1+\theta_2+\theta_3=1$, and $0<p_1, p_2, p \leq \infty$ satisfying
	\begin{align*}
		\frac{1}{p_1}=\frac{\theta_1}{p_{1,1}}+\frac{\theta_2}{p_{1,2}}+\frac{\theta_3}{p_{1,3}}, \quad \frac{1}{p_2}=\frac{\theta_1}{p_{2,1}}+\frac{\theta_2}{p_{2,2}}+\frac{\theta_3}{p_{2,3}}, \quad \frac{1}{p}=\frac{\theta_1}{p_1}+\frac{\theta_2}{p_2}+\frac{\theta_3}{p_3},
	\end{align*}
	we have
	\[\left\|T\left(f_1, f_2\right)\right\|_{L^{p, \infty}\left(\mathbb{R}^d\right)} \lesssim M_1^{\theta_1} M_2^{\theta_2} M_3^{\theta_3}\left\|f_1\right\|_{L^{p_1}\left(\mathbb{R}^d\right)}\left\|f_2\right\|_{L^{p_2}\left(\mathbb{R}^d\right)}.\]
	Moreover, if the convex hull of $\left\{\left(\frac{1}{p_{1,k}},\frac{1}{p_{2,k}}\right),\;k=1,2,3\right\}$ forms a non-trivial triangle in $\mathbb{R}^2$, then
	\[\left\|T\left(f_1, f_2\right)\right\|_{L^p\left(\mathbb{R}^d\right)} \lesssim M_1^{\theta_1} M_2^{\theta_2} M_3^{\theta_3}\left\|f_1\right\|_{L^{p_1}\left(\mathbb{R}^d\right)}\left\|f_2\right\|_{L^{p_2}\left(\mathbb{R}^d\right)}.\]
\end{theorem}
The following complex interpolation theorem can be found in \cite[Corollary 7.2.11]{GrafakosmodernFA}.
\begin{theorem}[\cite{GrafakosmodernFA}]\label{complexinterpolation}
	Let $\mathcal{T}$ be an m-linear operator defined on the m-fold product of spaces of finitely simple functions of $\sigma$-finite measure spaces $\left(X_i, \mu_i\right)$ and taking values in the set of measurable functions of another $\sigma-$finite measure space $(Y,\nu)$. Let $1 \leq p_{0, k}, p_{1, k}\leq \infty$ for all $1 \leq k \leq m+1, 0<\theta<1$, such that
	\[\frac{1}{p_k}=\frac{1-\theta}{p_{0, k}}+\frac{\theta}{p_{1, k}},\quad\text{for all}\;k \in\{1, \ldots, m+1\}.\]
	Suppose that for all finitely simple functions $f_k$ on $X_k$ we have
	\begin{align*}
		&\left\|\mathcal{T}\left(f_1, \ldots, f_m\right)\right\|_{L^{p_{0,m+1}}} \leq M_0\left\|f_1\right\|_{L^{p_{0,1}}} \cdots\left\|f_m\right\|_{L^{p_{0, m}}}, \\
		&\left\|\mathcal{T}\left(f_1, \ldots, f_m\right)\right\|_{L^{p_{1,m+1}}} \leq M_1\left\|f_1\right\|_{L^{p_{1,1}}} \cdots\left\|f_m\right\|_{L^{p_{1, m}}}
	\end{align*}
	Then for all finitely simple functions $f_k$ on $X_k$ we have
	\[\left\|\mathcal{T}\left(f_1, \ldots, f_m\right)\right\|_{L^{p_{m+1}}} \leq M_0^{1-\theta} M_1^\theta\left\|f_1\right\|_{L^{p_1}} \cdots\left\|f_m\right\|_{L^{p_m}}.\]
	Moreover, when $p_1, \ldots, p_m<\infty$, the operator $\mathcal{T}$ admits a unique bounded extension from $L^{p_1}\left(X_1, \mu_1\right) \times \cdots \times L^{p_m}\left(X_m, \mu_m\right)$ to $L^{p_{m+1}}(Y, \nu)$.
\end{theorem}
We will require the following lemma from \cite[Lemma 7.5.2(b)]{GrafakosmodernFA} to deal with medium frequency terms.
\begin{lemma}[\cite{GrafakosmodernFA}]\label{lemma:GrafMedium}
	Let $F_j$ be $L^2$ functions that satisfy $\sum_{j \in \Z}\left\|F_j\right\|_{L^2}^2<\infty$. Suppose that the Fourier transforms of $F_j$ are supported in the annulus $2^{j+b_1} \leq|\xi| \leq 2^{j+b_2}$. Then for any $0<p<\infty$ there is a constant $C>0$ (depending on $d,p,b_1,b_2$) such that
	\[\left\|\sum_{j \in \Z} F_j\right\|_{L^p} \leq C\left\|\left(\sum_{j \in \Z}\left|F_j\right|^2\right)^{\frac{1}{2}}\right\|_{L^p}.\]
\end{lemma}
We will also require the following improving property of local bilinear operators. The following lemma is from \cite[Proposition 4.1]{IPS}.
\begin{lemma}[\cite{IPS}]\label{lemma:local}
	Let $T$ be a bilinear operator satisfying the following.
	\begin{enumerate}
		\item There exists $N>0$ such that for functions $f_i, f_j$ supported in the unit cubes $Q_{l_i}$ and $Q_{l_j}$ with their lower left corners at $l_i$ and $l_j$ respectively and $\left\|l_i-l_j\right\|_{\infty}>N$, we have $T\left(f_i, f_j\right)=0$.
		\item There exists $R>0$ such that $T\left(f_1, f_2\right)$ is supported in the union of sets $\left(\operatorname{supp}\left(f_1\right)+B(0, R)\right)$ and $\left(\operatorname{supp}\left(f_2\right)+B(0, R)\right)$.
		\item The operator $T$ satisfies the estimate
		\[\left\|T\left(f_1, f_2\right)\right\|_1 \leq C\left\|f_1\right\|_{p_1}\left\|f_2\right\|_{p_2},\]
		for all functions $f_1, f_2$ supported in a fixed cube and for some fixed exponents $p_1$ and $p_2$ with $p_1, p_2 \geq 1, \frac{1}{p_1}+\frac{1}{p_2}=\frac{1}{p}>1$.
	\end{enumerate}
	Then, for all $r \in[p, 1]$, we have
	\[\|T\|_{L^{p_1} \times L^{p_2} \rightarrow L^r} \lesssim C.\]
\end{lemma}

\section{Estimates for single scale operator}\label{sec:singlescale}
In this section, we prove $L^p-$bounds for the single scale operator $T_\Omega^0$, which will be used to obtain similar bounds for its multiscale counterpart.
\begin{proposition}
	Let $d\geq1$, $1<q\leq\infty$ and $1\leq p_1,p_2\leq\infty$ with $\frac{1}{p_1}+\frac{1}{p_2}=2-\frac{1}{q}=\frac{1}{p}$, and $\eta$ be a bounded, compactly supported function. Then, we have the following estimates.
		\begin{align}
			\|T^0_\Omega(f_1,f_2)\eta\|_{L^1(\R^d)}&\lesssim\|\Omega\|_{L^q(\Sp^{2d-1})}\|f_1\|_{{L^{p_1}(\R^d)}}\|f_2\|_{{L^{p_2}(\R^d)}},\label{singlescale:L1bound}\\
			\|T^0_\Omega(f_1,f_2)\eta\|_{L^{2q}(\R^d)}&\lesssim\|\Omega\|_{L^q(\Sp^{2d-1})}\|f_1\|_{{L^{(2q)'}(\R^d)}}\|f_2\|_{{L^{\infty}(\R^d)}}\label{singlescale:2qbound1},\\
			\|T^0_\Omega(f_1,f_2)\eta\|_{L^{2q}(\R^d)}&\lesssim\|\Omega\|_{L^q(\Sp^{2d-1})}\|f_1\|_{{L^{\infty}(\R^d)}}\|f_2\|_{{L^{(2q)'}(\R^d)}}\label{singlescale:2qbound2}.
		\end{align}
		Consequently, we have that
		\begin{equation}\label{singlescale:Lpbound}
			\|T^0_\Omega(f_1,f_2)\|_{{L^{p}(\R^d)}}\lesssim\|\Omega\|_{L^q(\Sp^{2d-1})}\|f_1\|_{{L^{p_1}(\R^d)}}\|f_2\|_{{L^{p_2}(\R^d)}}.
		\end{equation}
\end{proposition}
\begin{proof}
	Consider the trilinear operator $\mathcal{T}$ defined as follows:
	\[\mathcal{T}(\Omega,f_1,f_2)(x):=T^0_\Omega(f_1,f_2)(x)\eta(x).\]
	We observe that for $q=1$ and $q=\infty$, the following bounds hold. 
	\begin{align*}
		\|\mathcal{T}(\Omega,f_1,f_2)\|_1\lesssim&\|\Omega\|_\infty\|f_1\|_1\|f_2\|_1\quad\text{and}\\
		\|\mathcal{T}(\Omega,f_1,f_2)\|_1\lesssim&\|\Omega\|_1\|f_1\|_{r_1}\|f_2\|_{r_2}\quad\text{for}\;\frac{1}{r_1}+\frac{1}{r_2}=1.
	\end{align*}
	Note that the first estimate is a simple consequence of the following inequality 
	\[\|\mathcal{T}(\Omega,f_1,f_2)\|_1\leq\|\Omega\|_\infty\int|\eta(x)|\left(\int |f_1(x-y_1)|dy_1\int |f_2(x-y_2)|dy_2\right)dx\lesssim\|\Omega\|_\infty\|f_1\|_1\|f_2\|_1.\]
	The second estimate is obtained by applying H\"{o}lder's inequality as below
	\begin{align*}
		\|\mathcal{T}(\Omega,f_1,f_2)\|_1\leq&\int\left|\Omega\left(\frac{(y_1,y_2)}{|(y_1,y_2)|}\right)\right|\left(\int |f_1(x-y_1)|^{r_1}dx\right)^\frac{1}{r_1}\left(\int |f_2(x-y_2)|^{r_2}dx\right)^\frac{1}{r_2}dy_1dy_2\\
		\lesssim&\|\Omega\|_1\|f_1\|_{r_1}\|f_2\|_{r_2}.
	\end{align*}
	The estimates for $1<q<\infty$ follow by applying complex interpolation (\Cref{complexinterpolation}) to the operator $\mathcal{T}$.
	Next, observe that by the support property $\supp{K^0}\subset\{2^{-1}\leq|(y_1,y_2)|\leq2\}$ along with the above estimates, the operator $T^0_\Omega$ satisfies the hypothesis of \Cref{lemma:local}, and the inequality \eqref{singlescale:Lpbound} follows as a consequence of the same. 

	The inequalities \eqref{singlescale:2qbound1} and \eqref{singlescale:2qbound2} follow by interpolating between the estimates,
	\begin{eqnarray*}
		&\|\mathcal{T}(\Omega,f_1,f_2)\|_\infty\lesssim\|\Omega\|_\infty\|f_1\|_1\|f_2\|_\infty,\quad&\|\mathcal{T}(\Omega,f_1,f_2)\|_\infty\lesssim\|\Omega\|_\infty\|f_1\|_\infty\|f_2\|_1,\\
		&\|\mathcal{T}(\Omega,f_1,f_2)\|_2\lesssim\|\Omega\|_1\|f_1\|_2\|f_2\|_\infty,\quad&\|\mathcal{T}(\Omega,f_1,f_2)\|_2\lesssim\|\Omega\|_1\|f_1\|_\infty\|f_2\|_2.
	\end{eqnarray*}
	Note that the last two estimates in the above are simple consequences of Minkowski's integral inequality.
\end{proof}

\begin{proposition}
	Let $1\leq p_1,p_2\leq\infty$ with $\frac{1}{p_1}+\frac{1}{p_2}=\frac{1}{p}$. Then, we have 
	\begin{equation}\label{singlescale:Lpbound_d=1}
		\|T^0_\Omega(f_1,f_2)\|_{{L^{p}(\R)}}\lesssim\|\Omega\|_{L^1(\Sp^{1})}\|f_1\|_{{L^{p_1}(\R)}}\|f_2\|_{{L^{p_2}(\R)}},
	\end{equation}
	whenever $\Omega$ is supported in the set $\{(\theta_1,\theta_2)\in\Sp^1:|\theta_1-\theta_2|\geq\frac{1}{2}\}$.
\end{proposition}
\begin{proof}
	We claim that the following $(1,1,1)-$bound holds true:
	\[\|T_\Omega^0(f_1,f_2)\|_1\lesssim\|\Omega\|_1\|f_1\|_1\|f_2\|_1.\]
	Once we have the estimate as claimed above, the required $(1,1,\frac{1}{2})-$inequality follows by \Cref{lemma:local}. 
	
	In order to prove the claim, we use polar coordinates to write
	\begin{align*}
		\|T_\Omega^0(f_1,f_2)\|_1\lesssim&\;\int_{\R}\int_{r=\frac{1}{2}}^2\int_{|(\theta_1-\theta_2)|\geq\frac{1}{2}}|\Omega(\theta_1,\theta_2)|f_1(x-r\theta_1)||f_2(x-r\theta_2)|d\sigma(\theta_1,\theta_2)drdx\\
		\lesssim&\;\int_{|(\theta_1-\theta_2)|\geq\frac{1}{2}}|\Omega(\theta_1,\theta_2)|\int_{r=\frac{1}{2}}^2\int_{\R}|f_1(x-r(\theta_1-\theta_2))||f_2(x)|dxdrd\sigma(\theta_1,\theta_2)\\
		\lesssim&\;\int_{|(\theta_1-\theta_2)|\geq\frac{1}{2}}|\Omega(\theta_1,\theta_2)|\int_{r=\frac{(\theta_1-\theta_2)}{2}}^{2(\theta_1-\theta_2)}\int_{\R}|f_1(x-r)||f_2(x)|dx\frac{dr}{(\theta_1-\theta_2)}d\sigma(\theta_1,\theta_2)\\
		\lesssim&\;\int_{|(\theta_1-\theta_2)|\geq\frac{1}{2}}|\Omega(\theta_1,\theta_2)|\int_{\R}\int_{\R}|f_1(x-r)||f_2(x)|dxdrd\sigma(\theta_1,\theta_2)\\
		\lesssim&\;\|\Omega\|_1\|f_1\|_1\|f_2\|_1.
	\end{align*}
\end{proof}

\section{Proof of \Cref{thm:Smoothing}}\label{sec:proofsmoothing}
We first prove the inequality \eqref{inq:Trilinearsmoothing1} whenever $f_1,f_2,f_3$ are functions such that $\supp{\widehat{f_{1}}}\cup\supp{\widehat{f_{2}}}\subset\left\{\lambda\leq|\xi|\leq2\lambda\right\}$ and $\supp{\widehat{f_{3}}}\subset\left\{|\xi|\leq2\lambda\right\}$. Let $\phi,\psi\in \mathcal{S}(\R^d)$ be even functions such that 
\begin{align*}
	\widehat\phi(\xi)=1\;\text{for}\;|\xi|\leq\frac{17}{8}&\quad\text{and}\;\supp{\widehat\phi}\subset \left\{|\xi|\leq\frac{9}{4}\right\},\\
	\widehat\psi(\xi)=1\;\text{for}\;\frac{7}{8}\leq|\xi|\leq\frac{17}{8}&\quad\text{and}\;\supp{\widehat\psi}\subset \left\{\frac{3}{4}\leq|\xi|\leq\frac{9}{4}\right\}.
\end{align*} 
For every $t>0$, we define $\phi_t(x)=t^d\phi(tx)$ and $\psi_t(x)=t^d\psi(tx)$. We note that using the Fourier supports of $f_1,f_2$, and $f_3$, we have
\[\langle T^0_\Omega(f_1,f_2),f_3\rangle=\langle T^0_\Omega(f_1*\psi_\lambda,f_2*\psi_\lambda),f_3*\phi_{\lambda}\rangle.\]
\subsection{Space localization of the operator}
Let $\eta\in C^\infty(\R^d)$ be supported in the cube $(-1,1)^d$ such that 
\begin{equation}\label{eqn:parofunity}
	\sum_{m\in\Z^d}\eta_m(x)=\sum_{m\in\Z^d}\eta(x-m)=1.
\end{equation}
Also, let $\widetilde{\eta}\in C_c^\infty(\R^d)$ be such that $\widetilde{\eta}(x)=1$ for $x\in[-4,4]^d$ and supported in $\left[-5,5\right]^d$. We set $\widetilde{\eta}_m(x)=\widetilde{\eta}(x-m)$. Since the kernel of $T^0_\Omega$ is compactly
supported in the annulus $\{y\in\R^{2d}:\;2^{-1}\leq|y|\leq2\}$, the spatial support is propagated in the trilinear form, in the sense that a localization of the function $f_3$ induces a localization on the functions $f_1$ and $f_2$. Thus, we have
\begin{align*}
	&\left|\left\langle T^0_\Omega\Big(f_1*\psi_\lambda,f_2*\psi_\lambda\Big),f_3*\phi_{\lambda}\right\rangle\right|\\
	\leq&\sum_{m\in\Z^d}\left|\left\langle T^0_\Omega\Big(f_1*\psi_\lambda,f_2*\psi_\lambda\Big),\eta_m(f_3*\phi_{\lambda})\right\rangle\right|\\
	\leq&\sum_{m\in\Z^d}\left|\left\langle T^0_\Omega\Big(\widetilde\eta_m^3(f_1*\psi_\lambda),\widetilde\eta_m^3(f_2*\psi_\lambda)\Big),\eta_m(f_3*\phi_{\lambda})\right\rangle\right|\\
	\leq&\sum_{m\in\Z^d}\left|\left\langle T^0_\Omega\Big(\widetilde\eta_m^2((\widetilde\eta_mf_1)*\psi_\lambda),\widetilde\eta_m^2((\widetilde\eta_mf_2)*\psi_\lambda)\Big),\eta_m(f_3*\phi_{\lambda})\right\rangle\right|\\
	&+\sum_{m\in\Z^d}\left|\left\langle T^0_\Omega\Big(\widetilde\eta_m^2(\widetilde\eta_m(f_1*\psi_\lambda)-(\widetilde\eta_mf_1)*\psi_\lambda),\widetilde\eta_m^2((\widetilde\eta_mf_2)*\psi_\lambda)\Big),\eta_m(f_3*\phi_{\lambda})\right\rangle\right|\\
	&+\sum_{m\in\Z^d}\left|\left\langle T^0_\Omega\Big(\widetilde\eta_m^2((\widetilde\eta_mf_1)*\psi_\lambda),\widetilde\eta_m^2(\widetilde\eta_m(f_2*\psi_\lambda)-(\widetilde\eta_mf_2)*\psi_\lambda)\Big),\eta_m(f_3*\phi_{\lambda})\right\rangle\right|\\
	=:&I_1+I_2+I_3.
\end{align*}
The estimates for the commutator terms $I_2$ and $I_3$ are similar, so we only discuss the term $I_2$. By the mean value theorem, we have that
\begin{align*}
	&\left|(\widetilde\eta_m(f_1*\psi_\lambda)-(\widetilde\eta_mf_1)*\psi_\lambda)(x)\right|\\
	\leq&\int|(\widetilde\eta_m(x)-\widetilde\eta_m(y))f_1(y)\psi_\lambda(x-y)|dy\\
	\lesssim&\lambda^{-1}\int\lambda|x-y||\psi_\lambda(x-y)f_1(y)|dy\\
	\leq&\lambda^{-1}|f_1|*\Psi_\lambda(x),
\end{align*}
where $\Psi(y)=|y||\psi(y)|$.
Using the above estimate and Cauchy-Schwarz inequality we obtain
\begin{align*}
	I_2\lesssim&\lambda^{-1}\|f_3\|_\infty\int\int|K^0(y_1,y_2)|(|f_1|*\Psi_\lambda(x-y_1))(|f_2|*|\psi_\lambda|(x-y_2))\sum|\eta_m(x)|dy_1dy_2dx\\
	\lesssim&\lambda^{-1}\|K^0\|_1\||f_1|*\Psi_\lambda\|_2\||f_2|*|\psi_\lambda|\|_2\|f_3\|_\infty\\
	\leq&\lambda^{-1}\|\Omega\|_1\|f_1\|_2\|f_2\|_2\|f_3\|_\infty.
\end{align*}
Thus, it remains to estimate the term $I_1$. By an application of Cauchy-Schwarz inequality, it is enough to establish that there exists $c>0$ such that the inequality
\[\left|\left\langle T^0_\Omega\Big(\widetilde\eta_m^2((\widetilde\eta_mf_1)*\psi_\lambda),\widetilde\eta_m^2((\widetilde\eta_mf_2)*\psi_\lambda)\Big),\eta_m(f_3*\phi_{\lambda})\right\rangle\right|\lesssim\lambda^{-c}\|\Omega\|_q\|\widetilde\eta_m f_1\|_2\|\widetilde\eta_m f_2\|_2\|f_3\|_\infty,\]
holds for all $\Omega\in L^q(\Sp^{2d-1})$, $f_1,f_2\in L^2(\R^d)$ and $f_3\in L^\infty(\R^d)$. Moreover, by translation invariance it is enough to consider the case $m=0$.
\subsection{Reduction to an $(\infty,\infty,\infty)-$inequality}
\sloppy By inequality \eqref{singlescale:L1bound}, we have the following inequality without any decay in the frequency parameter $\lambda$.
\begin{equation}\label{est:singlescale1}
	\left|\left\langle T^0_\Omega\Big(\widetilde\eta^2((\widetilde\eta f_1)*\psi_\lambda),\widetilde\eta^2((\widetilde\eta f_2)*\psi_\lambda)\Big),\eta(f_3*\phi_{\lambda})\right\rangle\right|\lesssim\|\Omega\|_{\frac{q+1}{2}}\|\widetilde\eta f_1\|_{(q+1)'}\|\widetilde\eta f_2\|_{(q+1)'}\|f_3\|_\infty.
\end{equation}
Thus, by the above inequality and complex interpolation (\Cref{complexinterpolation}), it is enough to show that there exists $c>0$ such that
\begin{equation}\label{est:infty^4}
	\left|\left\langle T^0_\Omega\Big(\widetilde\eta^2(f_1*\psi_\lambda),\widetilde\eta^2(f_2*\psi_\lambda)\Big),\eta(f_3*\phi_{\lambda})\right\rangle\right|\lesssim\lambda^{-c}\|\Omega\|_\infty\|f_1\|_{\infty}\|f_2\|_{\infty}\|f_3\|_\infty,
\end{equation}
holds for all $\Omega\in L^\infty(\Sp^{2d-1})$, $f_1,f_2\in L^\infty([-5,5]^d)$ and $f_3\in L^\infty(\R^d)$.
\subsection{Fourier series expansion of functions} We begin by defining
\[f_{l,\lambda}(x)=\begin{cases}\widetilde\eta(x)(f_l*\psi_\lambda)(x),&l=1,2,\\\eta(x)(f_l*\phi_\lambda)(x),&l=3.\end{cases}\]

It is easy to see that $\left\|\frac{\partial^\alpha}{\partial x^\alpha} (f_{l,\lambda})\right\|_\infty\lesssim\lambda^{|\alpha|_1}\|f_l\|_\infty$, for $l=1,2,3$ and any $\alpha\in(\N\cup\{0\})^d$. By a further smooth partition of unity into finite number (depending only on the dimension $d$) of terms and translation invariance, we may assume $\eta$ and $\widetilde{\eta}$ to be supported in the cube $\left[-\frac{1}{2},\frac{1}{2}\right]^d$. We expand the functions $f_{l,\lambda}$ into their Fourier series as below
\begin{equation*}
	\widetilde\eta(x) f_{l,\lambda}(x)=\widetilde\eta(x)\sum_{k_l\in\Z^d}a_{l,k_l}e^{-2\pi i k_l\cdot x},\quad\text{where}\quad a_{l,k_l}=\int_{\R^d}f_{l,\lambda}(x)e^{2\pi i k_l\cdot x}dx.
\end{equation*}
We now state some properties of the Fourier coefficients obtained above. The first is simply Plancherel's theorem while the remaining two estimates essentially quantify the Fourier support of the functions in terms of their respective Fourier coefficients. We have the following result:
\begin{lemma}\label{lemma:Fouriercoeff}
	Let $0<\epsilon_1<1$. Then we have the following estimates
	\begin{align}
		\|a_{l,k_l}\|_{\ell_2(k_l\in\Z^d)}^2\lesssim&\|f_{l,\lambda}\|_2^2\lesssim\|f_l\|_\infty^2,\quad\;l=1,2,3.\label{est:FourierBessel}\\
		\|a_{l,k_l}\|_{\ell_1(|k_l|>\lambda^{1+\epsilon_1})}\lesssim&\lambda^{-100}\|f_l\|_\infty,\quad\;l=1,2,3.\label{est:Fouriertail}\\
		\|a_{l,k_l}\|_{\ell_2(|k_l|\leq\frac{3}{8}\lambda)}\lesssim&\lambda^{-100}\|f_l\|_\infty,\quad\;l=1,2.\label{est:Fouriermain}
	\end{align}
\end{lemma}
\begin{proof}
	The estimate \eqref{est:FourierBessel} follows from Bessel's inequality (see, for instance, \cite[Proposition 3.2.6]{GrafakosclassicalFA}).

	To obtain \eqref{est:Fouriertail}, we first observe that by multiple applications of integration by parts, we have that
	\[|a_{l,k_l}|\lesssim\frac{\lambda^{|\alpha|_1}}{|k_l^\alpha|}\|f_l\|_\infty,\quad\;\text{for all}\;\alpha\in(\N\cup\{0\})^d\;\text{and}\;l=1,2,3.\]
	We simply choose $\alpha=(N,N,\cdots,N)$ for $N\in\N$ such that $\epsilon_1N>100$ and the estimate \eqref{est:Fouriertail} follows by summing over $k_l$.

	To obtain \eqref{est:Fouriermain}, we note that for $|k_1|\leq\frac{3}{8}\lambda$ and $|\xi|\geq\frac{3\lambda}{4}$, we have
	\begin{align*}
		|k_1+\xi|\geq&\;|\xi|-|k_1|
		\geq\frac{3}{4}\lambda-\frac{3}{8}\lambda=\frac{3}{8}\lambda.
	\end{align*}
	Hence, the above estimate along with $\widehat{\widetilde\eta}\in\mathcal{S}(\R^d)$ implies that for $|k_1|\leq\frac{3}{8}\lambda$, we have
	\begin{align*}
		|a_{1,k_1}|=&\;\left|\widehat{\widetilde\eta}*(\widehat{f}_1\widehat{\psi}_\lambda)(- k_1)\right|\\
		\lesssim&\;\left|\int_{\frac{3\lambda}{4}\leq|\xi|\leq\frac{9\lambda}{4}}\widehat{\eta}(-k_1-\xi)\widehat{f}_1(\xi)\widehat{\psi}_\lambda(\xi)d\xi\right|\\
		\lesssim&\;\|(f_1*\psi_\lambda)^\wedge\|_\infty\int_{\frac{3\lambda}{4}\leq|\xi|\leq\frac{9\lambda}{4}}\frac{d\xi}{|k_1+\xi|^N}\\
		\lesssim&\;\|f_1*\psi_\lambda\|_1 \lambda^{d-N}\lesssim\|f_1\|_\infty \lambda^{d-N}.
	\end{align*}
	This in turn implies that $\|a_{1,k_1}\|_{\ell_2(|k_1|\leq\frac{3}{8}\lambda)}\lesssim\lambda^{\frac{3d}{2}-N}\|f_1\|_\infty$. Choosing $N\in\N$ large enough gives the estimate \eqref{est:Fouriermain}. The estimate for $l=2$ is similar.
\end{proof}

Next, we substitute the Fourier series expansions of functions into the left hand side of the expression \eqref{est:infty^4} to obtain:
\begin{align*}
	&\left|\left\langle T^0_\Omega\Big(\widetilde\eta^2(f_1*\psi_\lambda),\widetilde\eta^2(f_2*\psi_\lambda)\Big),\eta(f_3*\phi_{\lambda})\right\rangle\right|=\left|\sum_{\vec{\mathbf{k}}=(k_1,k_2,k_3)\in\Z^{3d}}\prod_{l=1}^{3}a_{l,k_l}I_{\vec{\mathbf{k}}}\right|,
\end{align*}
where the quantity $I_{\vec{\mathbf{k}}},\;\vec{\mathbf{k}}=(k_1,k_2,k_3)$ is given by
\[\int_{\R^d}\int_{\{2^{-1}\leq|(y_1,y_2)\leq2\}}K^0(y_1,y_2)\widetilde\eta(x-y_1)\widetilde\eta(x-y_2)\overline{\widetilde\eta(x)}e^{-2\pi i(k_1,k_2,k_3)\cdot (x-y_1,x-y_2,-x)}dy_1dy_2dx.\]

Thus using $|I_{\vec{\mathbf{k}}}|\lesssim\|\Omega\|_1$ and \Cref{lemma:Fouriercoeff}, we have that
\begin{align}\label{reducedoperator}
	&\left|\langle T^0_\Omega\Big(\widetilde\eta^2(f_1*\psi_\lambda),\widetilde\eta^2(f_2*\psi_\lambda)\Big),\eta(f_3*\phi_{\lambda})\rangle\right|\nonumber\\
	&\leq\Bigg|\sum_{\substack{\vec{\mathbf{k}}\in\Z^{3d}\\\frac{3}{8}\lambda\leq|k_1|\leq\lambda^{1+\epsilon_1},\\\frac{3}{8}\lambda\leq|k_2|\leq\lambda^{1+\epsilon_1},\\|k_3|\leq\lambda^{1+\epsilon_1}}}\prod_{l=1}^{3}a_{l,k_l}I_{\vec{\mathbf{k}}}\Bigg|	+\;O(\lambda^{-50}\|\Omega\|_1\|f_1\|_\infty\|f_2\|_\infty\|f_3\|_\infty).
\end{align}
To see this assertion, we note that
\begin{align*}
	&\Bigg|\langle T^0_\Omega\Big(\widetilde\eta^2(f_1*\psi_\lambda),\widetilde\eta^2(f_2*\psi_\lambda)\Big),\eta(f_3*\phi_{\lambda})\rangle-\sum_{\substack{\vec{\mathbf{k}}\in\Z^{3}\\\frac{3}{8}\lambda\leq|k_1|\leq\lambda^{1+\epsilon_1},\\\frac{3}{8}\lambda\leq|k_2|\leq\lambda^{1+\epsilon_1},\\|k_3|\leq\lambda^{1+\epsilon_1}}}\prod_{l=1}^{3}a_{l,k_l}I_{\vec{\mathbf{k}}}\Bigg|\\
	\leq&\Bigg|\sum_{\substack{\frac{3}{8}\lambda\leq|k_1|\leq\lambda^{1+\epsilon_1},\\\frac{3}{8}\lambda\leq|k_2|\leq\lambda^{1+\epsilon_1},\\|k_3|>\lambda^{1+\epsilon_1}}}\prod_{l=1}^{3}a_{l,k_l}I_{\vec{\mathbf{k}}}\Bigg|+\Bigg|\sum_{\substack{\frac{3}{8}\lambda\leq|k_1|\leq\lambda^{1+\epsilon_1},\\|k_2|>\lambda^{1+\epsilon_1},\\|k_3|\leq\lambda^{1+\epsilon_1}}}\prod_{l=1}^{3}a_{l,k_l}I_{\vec{\mathbf{k}}}\Bigg|+\Bigg|\sum_{\substack{|k_1|>\lambda^{1+\epsilon_1},\\\frac{3}{8}\lambda\leq|k_2|\leq\lambda^{1+\epsilon_1},\\|k_3|\leq\lambda^{1+\epsilon_1}}}\prod_{l=1}^{3}a_{l,k_l}I_{\vec{\mathbf{k}}}\Bigg|\\
	&+\Bigg|\sum_{\substack{|k_1|<\frac{3}{8}\lambda,\\\frac{3}{8}\lambda\leq|k_2|\leq\lambda^{1+\epsilon_1},\\|k_3|\leq\lambda^{1+\epsilon_1}}}\prod_{l=1}^{3}a_{l,k_l}I_{\vec{\mathbf{k}}}\Bigg|+\Bigg|\sum_{\substack{\frac{3}{8}\lambda\leq|k_1|\leq\lambda^{1+\epsilon_1},\\|k_2|<\frac{3}{8}\lambda,\\|k_3|\leq\lambda^{1+\epsilon_1}}}\prod_{l=1}^{3}a_{l,k_l}I_{\vec{\mathbf{k}}}\Bigg|,\\
	=:&\sum_{j=1}^5 A_j.
\end{align*}
Observe that applying Cauchy-Schwarz inequality along with estimates \eqref{est:FourierBessel} and \eqref{est:Fouriertail} gives us the following bound. 
\begin{align*}A_1&\leq\|\Omega\|_1\lambda^{1+\epsilon_1}\|a_{1,.}\|_2\|a_{2,.}\|_2\|a_{3,k_3}\|_{\ell_1\{|k_3|>\lambda^{1+\epsilon_1}\}}\\
	&\lesssim\lambda^{-50}\|\Omega\|_1\|f_1\|_\infty\|f_2\|_\infty\|f_3\|_\infty.
	\end{align*}
The desired bounds for the terms $A_j,\;j=2,3,4,5$ are proved similarly with the exception that we use \eqref{est:Fouriermain} instead of \eqref{est:Fouriertail} for the cases $j=4,5$.

Therefore, it remains to estimate the first term in the right hand side of \eqref{reducedoperator}.
\subsection{Oscillatory analysis of $I_{\vec{\mathbf{k}}}$:} We use polar coordinates to express $I_{\vec{\mathbf{k}}}$ as
\begin{equation}\label{polarI}
I_{\vec{\mathbf{k}}}=\int_{\Sp^{2d-1}}\Omega(\theta)\int_{\R^d}\int_{\frac{1}{2}}^2\zeta_{\theta}(x,r)e^{-2\pi i\mathcal{P}_{\vec{\mathbf{k}},\theta}(x,r)}drdxd\sigma(\theta),
\end{equation}
where the functions $\zeta_{\theta}$ and $\mathcal{P}_{\vec{\mathbf{k}},\theta}(x,r)$ are defined as
\begin{align*}
	\zeta_{\theta}(x,r)=&\;r^{-1}\beta(r)\widetilde\eta(x-r\theta_1)\widetilde\eta(x-r\theta_2)\overline{\widetilde\eta(x)},\\
	\mathcal{P}_{\vec{\mathbf{k}},\theta}(x,r)=&\;(k_1+k_2-k_3)\cdot x-r(k_1,k_2)\cdot(\theta_1,\theta_2),\quad\theta=(\theta_1,\theta_2).
\end{align*}
The functions $\zeta_{\theta}$ and $\mathcal{P}_{\vec{\mathbf{k}},\theta}(x,r)$ satisfy the following properties:
\begin{align}
	|\partial^\alpha_x\partial^N_r\zeta_{\theta}(x,r)|\lesssim&\;\chi_{B(0,2)}(x)\chi_{[\frac{1}{2},2]}(r),\quad\forall\;(\alpha,N)\in(\N\cup\{0\})^{d+1},\label{integrand}\\
	\nabla_{x,r}\mathcal{P}_{\vec{\mathbf{k}},\theta}(x,r)=&\;\big(k_1+k_2-k_3,(k_1,k_2)\cdot(\theta_1,\theta_2)\big).\label{phase}
\end{align}
Let $0<\epsilon_2<1$ be a real number to be determined later. Based on the constant gradient of the phase function $\mathcal{P}_{\vec{\mathbf{k}},\theta}$ in the space variable, we decompose the required quantity in \eqref{reducedoperator} into two parts as follows
\begin{align*}
	\Bigg|\sum_{\substack{\frac{3}{8}\lambda\leq|k_1|\leq\lambda^{1+\epsilon_1},\\\frac{3}{8}\lambda\leq|k_2|\leq\lambda^{1+\epsilon_1},\\|k_3|\leq\lambda^{1+\epsilon_1}}}\prod_{l=1}^{3}a_{l,k_l}I_{\vec{\mathbf{k}}}\Bigg|\leq&\;\Bigg|\sum_{\substack{\frac{3}{8}\lambda\leq|k_1|\leq\lambda^{1+\epsilon_1},\\\frac{3}{8}\lambda\leq|k_2|\leq\lambda^{1+\epsilon_1},\\|k_3|\leq\lambda^{1+\epsilon_1},\\|k_1+k_2-k_3|>\lambda^{\epsilon_2}}}\prod_{l=1}^{3}a_{l,k_l}I_{\vec{\mathbf{k}}}\Bigg|+\Bigg|\sum_{\substack{\frac{3}{8}\lambda\leq|k_1|\leq\lambda^{1+\epsilon_1},\\\frac{3}{8}\lambda\leq|k_2|\leq\lambda^{1+\epsilon_1},\\|k_3|\leq\lambda^{1+\epsilon_1},\\|k_1+k_2-k_3|\leq\lambda^{\epsilon_2}}}\prod_{l=1}^{3}a_{l,k_l}I_{\vec{\mathbf{k}}}\Bigg|\\
	=:\;&S_1+S_2.
\end{align*}

\subsection*{Estimate for $S_1$:} For a fixed $k_l=(k_{l1},\dots,k_{ld}),\;l=1,2,3$, we choose the index $j_k\in\{1,2,\dots,d\}$ such that $|k_{1j_k}+k_{2j_k}-k_{3j_k}|\geq|k_{1j}+k_{2j_k}-k_{3j}|$ for all $j\in\{1,2,\dots,d\}$. Thus we have $|k_1+k_2-k_3|\leq d^{\frac{1}{2}}|k_{1j_k}+k_{2j_k}-k_{3j_k}|$. We apply $N-$fold integration by parts in the $x_{j_k}-$variable to obtain
\begin{align*}
	S_1\lesssim\;&\sum_{\substack{\frac{3}{8}\lambda\leq|k_1|\leq\lambda^{1+\epsilon_1},\\\frac{3}{8}\lambda\leq|k_2|\leq\lambda^{1+\epsilon_1},\\|k_3|\leq\lambda^{1+\epsilon_1},\\|k_1+k_2-k_3|>\lambda^{\epsilon_2}}}\prod_{l=1}^{3}|a_{l,k_l}|\int_{\Sp^{2d-1}}|\Omega(\theta)|\int_{\R^d}\int_{\frac{1}{2}}^2\frac{|\partial_{x_{j_k}}^N\zeta_{\theta}(x,r)|}{(|k_{1j_k}+k_{2j_k}-k_{3j_k}|)^N}drdxd\sigma(\theta)\\
	\lesssim\;&\frac{1}{\lambda^{\epsilon_2N}}\prod_{l=1}^{3}\sum_{|k_l|\leq\lambda^{1+\epsilon_1}}|a_{l,k_l}|\int_{\Sp^{2d-1}}|\Omega(\theta)|\int_{\R^d}\int_{\frac{1}{2}}^2|\partial_{x_{j_k}}^N\zeta_{\theta}(x,r)|drdxd\sigma(\theta)\\
	\lesssim\;&\frac{\lambda^{\frac{3(1+\epsilon_1)d}{2}}}{\lambda^{\epsilon_2N}}\|\Omega\|_1\prod_{l=1}^{3}\|a_{l,k_l}\|_2\\
	\lesssim\;&\lambda^{\frac{3(1+\epsilon_1)d}{2}-\epsilon_2N}\|\Omega\|_1\|f_1\|_\infty\|f_2\|_\infty\|f_3\|_\infty
\end{align*}
where we used Cauchy-Schwarz inequality along with estimate \eqref{est:FourierBessel} in the second to the last inequality. Choosing $N\in\N$ large enough produces
\begin{equation*}
	S_1\lesssim\lambda^{-50}\|\Omega\|_1\|f_1\|_\infty\|f_2\|_\infty\|f_3\|_\infty.
\end{equation*}
\subsection*{Estimate for $S_2$:} We begin by observing the following:
\begin{equation}\label{est:ell2}
	\|I_{(\cdot,\cdot,k_3)}\|_{\ell_2(\Z^{2d})}\lesssim\|\Omega\|_\infty,\quad\text{for any}\;k_3\in\Z^d.
\end{equation}
Indeed, we use Minkowski's inequality and  Bessel's inequality for the family of functions $F_x,\;x\in\R^d$ given by 
\[F_x(y_1,y_2)=K^0(y_1,y_2)\widetilde{\eta}(x-y_1)\widetilde{\eta}(x-y_2),\]
to obtain,
\begin{align*}
	\|I_{(\cdot,\cdot,k_3)}\|_{\ell_2(\Z^{2d})}&\lesssim\int_{\R^d}|\widetilde{\eta}(x)|\left(\sum_{(k_1,k_2)\in\Z^{2d}}|\widehat{F_x}(k_1,k_2)|^2\right)^\frac{1}{2}dx\\
	&\lesssim\int_{\R^d}|\widetilde{\eta}(x)|\|F_x\|_2dx\\
	&\lesssim\|\Omega\|_\infty.
\end{align*}
However, the following decay inequality for $I_{\vec{\mathbf{k}}}$ is crucial for our analysis.
\begin{lemma}\label{lemma:ellinfty}
	For any $0<\epsilon_2<1$ and $k_1,k_2,k_3\in\Z^d$ with $|(k_1,k_2)|\geq\frac{3}{8}\lambda$, we have
	\[|I_{\vec{\mathbf{k}}}|\lesssim\lambda^{-(1-\epsilon_2)}\|\Omega\|_\infty.\]
\end{lemma}
\begin{proof}
	For a fixed $0<\epsilon_2<1$, we write $I_{\vec{\mathbf{k}}}=I_{\vec{\mathbf{k}}}^++I_{\vec{\mathbf{k}}}^-$, where
\begin{align*}
	I_{\vec{\mathbf{k}}}^+&=\int_{\{(\theta_1,\theta_2)\in\Sp^{2d-1}:|k_1.\theta_1+k_2.\theta_2|>\lambda^{\epsilon_2}\}}\Omega(\theta)\int_{\R^d}\int_{\frac{1}{2}}^2\zeta_{\theta}(x,r)e^{-2\pi i\mathcal{P}_{\vec{\mathbf{k}},\theta}(x,r)}drdxd\sigma(\theta),\\
	I_{\vec{\mathbf{k}}}^-&=I_{\vec{\mathbf{k}}}-I_{\vec{\mathbf{k}}}^+.
\end{align*}
For the term $I_{\vec{\mathbf{k}}}^+$, we apply $N-$fold integration by parts in the $r-$variable to obtain
\begin{align*}
	\left|I_{\vec{\mathbf{k}}}^+\right|&\lesssim \int_{\{|k_1.\theta_1+k_2.\theta_2|>\lambda^{\epsilon_2}\}}|\Omega(\theta)|\int_{\R^d}\int_{\frac{1}{2}}^2\frac{|\partial_{r}^N\zeta_{\theta}(x,r)|}{|k_1.\theta_1+k_2.\theta_2|^N}drdxd\sigma(\theta)\lesssim \frac{1}{\lambda^{\epsilon_2 N}}\|\Omega\|_1.
\end{align*}
We choose $N\in\N$ such that $\epsilon_2 N>100$.
We observe that the following elementary estimate holds. 
\begin{equation*}
	N(\omega,\varepsilon):=\sigma(\{\theta\in\Sp^{2d-1}:|\omega\cdot\theta|\leq\varepsilon\})\lesssim\varepsilon,\quad \text{for}\;\omega\in\Sp^{2d-1}.
\end{equation*}
Further, note that for $\frac{3}{8}\lambda\leq|(k_1,k_2)|$ and $|k_1\theta_1+k_2\theta_2|\leq\lambda^{\epsilon_2}$, we have that \[\left|\left(\frac{(k_1,k_2)}{|(k_1,k_2)|}\cdot(\theta_1,\theta_2)\right)\right|\lesssim\lambda^{-1+\epsilon_2}.\] 
In view of the above estimates, we obtain that
\begin{align*}
	|I_{\vec{\mathbf{k}}}^-|\lesssim\;&\int\limits_{\substack{\theta\in\Sp^{2d-1}:\\\left|\left(\frac{(k_1,k_2)}{|(k_1,k_2)|}\cdot(\theta_1,\theta_2)\right)\right|\lesssim\lambda^{-1+\epsilon_2}}}|\Omega(\theta)|\int_{\R^d}\int_{\frac{1}{2}}^2|\zeta_{\theta}(x,r)|drdxd\sigma(\theta)\\
	\lesssim\;&\|\Omega\|_\infty N\left(\frac{(k_1,k_2)}{|(k_1,k_2)|},4\lambda^{-1+\epsilon_2}\right)\\
	\lesssim\;&\lambda^{-(1-\epsilon_2)}\|\Omega\|_\infty.
\end{align*}
\end{proof}
We will also require the following counting lemma to estimate the sum of product of three Fourier coefficients over a fixed lattice by their product of $\ell_2$-norms with a bound depending on the number of lattice points.
\begin{lemma}\label{lemma:counting}
	Let $Q\subset\Z^{2d}$ with $|Q|<\infty$. Then, we have
	\[\sum_{(k_1,k_2)\in Q}\left|b_{1,k_1}b_{2,k_2}b_{3,k_1+k_2}\right|\lesssim |Q|^\frac{1}{4}\prod_{l=1}^3\|b_{l,k_l}\|_{\ell_2}. \]
\end{lemma}
\begin{proof}
	We define $Q_k=\{(k_1,k_2)\in Q:k_1+k_2=k\}$ and $E=\{k:|Q_k|>|Q|^\frac{1}{2}\}$. By Chebyshev's inequality, we have $|E|\leq |Q|^\frac{1}{2}$. We rewrite the desired sum as
	\begin{align*}
		\sum_{(k_1,k_2)\in Q}\left|b_{1,k_1}b_{2,k_2}b_{3,k_1+k_2}\right|&=\sum_{k\in E}\sum_{(k_1,k_2)\in Q_k}\left|b_{1,k_1}b_{2,k_2}b_{3,k}\right|+\sum_{k\notin E}\sum_{(k_1,k_2)\in Q_k}\left|b_{1,k_1}b_{2,k_2}b_{3,k_1+k_2}\right|\\
		&=\mathfrak{S}_1+\mathfrak{S}_2.
	\end{align*}
	By two fold applications of Cauchy-Schwarz inequality, we have
	\begin{align*}
		\mathfrak{S}_1&=\sum_{k\in E}|b_{3,k}|\sum_{(k_1,k_2)\in Q_k}|b_{1,k_1}b_{2,k_2}|\\
		&\leq\sum_{k\in E}|b_{3,k}|\|b_{1,k_1}\|_{\ell_2}\|b_{2,k_2}\|_{\ell_2}\\
		&\leq|E|^\frac{1}{2}\prod_{l=1}^3\|b_{l,k_l}\|_{\ell_2}\leq|Q|^\frac{1}{4}\prod_{l=1}^3\|b_{l,k_l}\|_{\ell_2}.
	\end{align*}
	By Cauchy-Schwarz inequality, we also have
	\begin{align*}
		\mathfrak{S}_2&=\sum_{k\notin E}\sum_{(k_1,k_2)\in Q_k}|b_{1,k_1}b_{2,k_2}||b_{3,k_1+k_2}|\\
		&\leq\|b_{1,k_1}\|_{\ell_2}\|b_{2,k_2}\|_{\ell_2}\left(\sum_{k\notin E}\sum_{(k_1,k_2)\in Q_k}|b_{3,k_1+k_2}|^2\right)^\frac{1}{2}\\
		&=\|b_{1,k_1}\|_{\ell_2}\|b_{2,k_2}\|_{\ell_2}\left(\sum_{k\notin E}|Q_k||b_{3,k}|^2\right)^\frac{1}{2}\leq|Q|^\frac{1}{4}\prod_{l=1}^3\|b_{l,k_l}\|_{\ell_2}.
	\end{align*}
\end{proof}
We now estimate $S_2$. We set $B=\{(k_1,k_2)\in\Z^{2d}:\;\frac{3}{8}\lambda\leq|k_1|,|k_2|\leq\lambda^{1+\epsilon_1}\}$ and $J_{k_1,k_2}^k=I_{k_1,k_2,k_1+k_2-k}$, for $k_1,k_2,k\in\Z^d$. For a fixed $k\in\Z^d\cap B(0,\lambda^{\epsilon_2})$ and $j\in\N\cup\{0\}$, we define the set $B_j^k$ as
\begin{align*}
	B_0^k&=\{(k_1,k_2)\in B:\;J_{k_1,k_2}^k=0\},\\
	B_j^k&=\{(k_1,k_2)\in B:\;2^{-j}\|J_{k_1,k_2}^k\|_{\ell_\infty((k_1,k_2)\in B)}<|J_{k_1,k_2}^k|\leq2^{-j+1}\|J_{k_1,k_2}^k\|_{\ell_\infty((k_1,k_2)\in B)}\},\;j\in\N.
\end{align*}
We note that $B=\cup_{j\in\N\cup\{0\}}B_j^k$ for a fixed $k$ and the following bound also holds true.
\begin{equation*}
	|B_j^k|\lesssim2^{2j}\|J_{k_1,k_2}^k\|_{\ell_2((k_1,k_2)\in B)}^2\|J_{k_1,k_2}^k\|_{\ell_\infty((k_1,k_2)\in B)}^{-2},\;j\in\N.
\end{equation*}
We use the above estimate along with \Cref{lemma:counting} to obtain
\begin{align*}
	S_2&\leq\sum_{k\in\Z^d\cap B(0,\lambda^{\epsilon_2})}\sum_{(k_1,k_2)\in B}|I_{k_1,k_2,k_1+k_2-k}a_{1,k_1}a_{2,k_2}a_{3,k_1+k_2-k}|\\
	&\leq\sum_{k\in\Z^d\cap B(0,\lambda^{\epsilon_2})}\sum_{j\in\N}\sum_{(k_1,k_2)\in B_j^k}|J_{k_1,k_2}^ka_{1,k_1}a_{2,k_2}a_{3,k_1+k_2-k}|\\
	&\lesssim\sum_{k\in\Z^d\cap B(0,\lambda^{\epsilon_2})}\sum_{j\in\N}\sum_{(k_1,k_2)\in B_j^k}2^{-j}\|J_{k_1,k_2}^k\|_{\ell_\infty((k_1,k_2)\in B)}|a_{1,k_1}a_{2,k_2}a_{3,k_1+k_2-k}|\\
	&\lesssim\sum_{k\in\Z^d\cap B(0,\lambda^{\epsilon_2})}\sum_{j\in\N}2^{-j}|B_j^k|^\frac{1}{4}\|J_{k_1,k_2}^k\|_{\ell_\infty((k_1,k_2)\in B)}\prod_{l=1}^3\|a_{l,k_l}\|_{\ell_2}\\
	&\lesssim\sum_{k\in\Z^d\cap B(0,\lambda^{\epsilon_2})}\sum_{j\in\N}2^{-\frac{j}{2}}\|J_{k_1,k_2}^k\|^\frac{1}{2}_{\ell_\infty((k_1,k_2)\in B)}\|J_{k_1,k_2}^k\|^\frac{1}{2}_{\ell_2((k_1,k_2)\in B)}\prod_{l=1}^3\|a_{l,k_l}\|_{\ell_2}\\
	&\lesssim\lambda^{-\frac{1-\epsilon_2}{2}+\epsilon_2d}\|\Omega\|_\infty\|f_1\|_\infty\|f_2\|_\infty\|f_3\|_\infty,
\end{align*}
where we used \eqref{est:ell2} and \Cref{lemma:ellinfty} in the last step.
Choosing $\epsilon_2>0$ sufficiently small concludes the proof of \eqref{est:infty^4}. Consequently, the estimate \eqref{inq:Trilinearsmoothing1} follows for the case when $\supp{\widehat{f_{1}}}\cup\supp{\widehat{f_{2}}}\subset\left\{\lambda\leq|\xi|\leq2\lambda\right\}$ and $\supp{\widehat{f_{3}}}\subset\left\{|\xi|\leq2\lambda\right\}$.

The proof of the estimate \eqref{inq:Trilinearsmoothing1} for the other symmetric cases follows by  minor modifications in the arguments used in the above proof. More precisely, we need to interchange the roles of $f_1,f_2,$ and $f_3$ accordingly and instead of using the single scale estimate \eqref{est:singlescale1}, we use the following estimates
\begin{align*}
	\left|\left\langle T^0_\Omega\Big(\widetilde\eta^2((\widetilde\eta f_1)*\psi_\lambda),\eta(f_2*\phi_{\lambda})\Big),\widetilde\eta^2((\widetilde\eta f_3)*\psi_\lambda)\right\rangle\right|&\lesssim\|\Omega\|_{\frac{q+1}{2}}\|\widetilde\eta f_1\|_{(q+1)'}\|f_2\|_{\infty}\|\widetilde\eta f_3\|_{(q+1)'},\\
	\left|\left\langle T^0_\Omega\Big(\eta(f_1*\phi_{\lambda}),\widetilde\eta^2((\widetilde\eta f_2)*\psi_\lambda)\Big),\widetilde\eta^2((\widetilde\eta f_3)*\psi_\lambda)\right\rangle\right|&\lesssim\|\Omega\|_{\frac{q+1}{2}}\|f_1\|_\infty\|\widetilde\eta f_2\|_{(q+1)'}\|\widetilde\eta f_3\|_{(q+1)'}.
\end{align*}
These estimates follow easily from \eqref{singlescale:2qbound1} and \eqref{singlescale:2qbound2}.
\qed
\section{Proof of \Cref{Thm:Rough}}\label{sec:proofrough}
We begin by introducing a spatial decomposition of the kernel. Let $\beta\in C_c^\infty{(\R)}$ be a radial function such that $\supp{\beta}\subset\left(\frac{1}{2},2\right)$ and $\sum_{i\in\Z}\beta_i(t)=\sum_{i\in\Z}\beta(2^it)=1$. We define the operators $T^i_\Omega,\;i\in\Z$ corresponding to the kernels 
\[K^i(y_1,y_2)=\Omega\left(\frac{(y_1,y_2)}{|(y_1,y_2)|}\right){|(y_1,y_2)|^{-2d}}\beta_i(|(y_1,y_2)|),\]
as follows 
\[T^i_\Omega(f_1,f_2)(x)=\int_{\R^{2d}}K^i(y_1,y_2)f_1(x-y_1)f_2(x-y_2)\;dy_1dy_2.\]
We further decompose the operator based on Littlewood-Paley decomposition of the functions in the following manner.
Let $\phi\in\mathcal{S}(\R^d)$ be a radial function such that $\phi$ is Fourier supported on the ball $B(0,2)$ with $\widehat\phi(\xi)=1,\;\xi\in B(0,1)$ and $0\leq\widehat\phi\leq1$. Let $\widehat\psi(\xi)=\widehat\phi(\xi)-\widehat\phi(2\xi)$. For $k\in\Z$, we define the sequence of functions $\phi_k$ and $\psi_k$ as $\phi_k(x)=2^{kd}\phi(2^{k}x)$ and $\psi_k(x)=2^{kd}\psi(2^{k}x)$. We have the following smooth partition of the identity $\widehat{\phi_0}(\xi)+\sum_{k\in\N}\widehat{\psi_k}(\xi)=1$ for $\xi\neq0$. Thus, we have
\begin{align}
	T_\Omega(f_1,f_2)(x)=&\;T_{\Omega}^{LL}(f_1,f_2)(x)+\sum_{k\in\N}T_{\Omega,k}^{HL}(f_1,f_2)(x)+\sum_{k\in\N}T_{\Omega,k}^{LH}(f_1,f_2)(x)+\sum_{k\in\N}T_{\Omega,k}^{HH}(f_1,f_2)(x),
\end{align}
where 
	\begin{align}T_{\Omega}^{LL}(f_1,f_2)(x)=&\;\sum_{i\in\Z}T^i_\Omega(f_1*\phi_i,f_2*\phi_i)(x),\nonumber\\
	T_{\Omega,k}^{HL}(f_1,f_2)(x)=&\;\sum_{i\in\Z}T^i_\Omega(f_1*\psi_{i+k},f_2*\phi_{i+k-10})(x),\nonumber\\
	T_{\Omega,k}^{LH}(f_1,f_2)(x)=&\;\sum_{i\in\Z}T^i_\Omega(f_1*\phi_{i+k-10},f_2*\psi_{i+k})(x)~\;\text{and}\nonumber\\
	T_{\Omega,k}^{HH}(f_1,f_2)(x)=&\;\sum_{i\in\Z}\sum_{a=-10}^{10}T^i_\Omega(f_1*\psi_{i+k+a},f_2*\psi_{i+k})(x).\nonumber
\end{align}
\subsection*{Estimate of Low frequency term :} Note that we have 
\[T_{\Omega}^{LL}\left(f_1, f_2\right)(x)=K_{LL} *(f_1 \otimes f_2)(x,x) \text {, where }\]
\[K_{LL}(y, z)=\sum_{i \in \mathbb{Z}} K^i *\left(\phi_i \otimes \phi_i\right)(y, z).\]
The kernel $K_{LL}$ satisfies the following Fourier transform estimate.
\begin{lemma}For all multi-indices $\alpha\in(\N\cup\{0\})^{2d}$, the following estimate holds
	\[|\partial^\alpha\widehat{K_{LL}}(\xi_1,\xi_2)|\lesssim\|\Omega\|_1|(\xi_1,\xi_2)|^{-|\alpha|}.\]
\end{lemma}
\begin{proof}
By the mean zero condition $\int_{\mathbb S^{2d-1}}\Omega=0$ and the mean value theorem, it is easy to verify that \[|\partial^\alpha\widehat{K^i}(\xi_1,\xi_2)|\lesssim\|\Omega\|_1 2^{-i|\alpha|}|2^{-i}(\xi_1,\xi_2)|.\] 
Thus we have
\begin{align*}
	|\partial^\alpha\widehat{K_{LL}}(\xi_1,\xi_2)|\lesssim&\;\sum_{i\in\Z}\sum_{|\alpha_1|+|\alpha_2|=|\alpha|}|\partial^{\alpha_1}\widehat{K^i}(\xi_1,\xi_2)||\partial^{\alpha_2}(\widehat{\phi_i}(\xi_1) \widehat{\phi_i}(\xi_2))|\\
	\lesssim&\;\|\Omega\|_1\sum_{i\in\Z}\sum_{|\alpha_1|+|\alpha_2|=|\alpha|}2^{-i|\alpha|}|2^{-i}(\xi_1,\xi_2)|2^{-i|\alpha_2|}\chi_{|2^{-i}(\xi_1,\xi_2)|\leq\sqrt{2}}(\xi_1,\xi_2)\\
	\lesssim&\;\|\Omega\|_1|(\xi_1,\xi_2)|^{-|\alpha|}\sum_{i\in\Z:|2^{-i}(\xi_1,\xi_2)|\leq\sqrt{2}}|2^{-i}(\xi_1,\xi_2)|\\
	\lesssim&\;\|\Omega\|_1|(\xi_1,\xi_2)|^{-|\alpha|}.
\end{align*}
The desired $L^p-$boundedness of the operator $T_{\Omega}^{LL}$ follows from the above lemma by applying the Coifman-Meyer multiplier theorem from \cite{CM1978}. We state their result below for convenience. The reader may also refer to \cite[Theorem 7.5.3]{GrafakosmodernFA} for a proof. 
\end{proof}
\begin{theorem}{\cite{CM1978}}
	Let $1<p_1, p_2<\infty,\;\frac{1}{2}< p<\infty$ with $\frac{1}{p_1}+\frac{1}{p_2}=\frac{1}{p}$. Let $m \in L^{\infty}\left(\mathbb{R}^{2 d}\right)$ satisfies
	\[\left|\partial^\alpha m\left(\xi_1, \xi_2\right)\right|\leq C\left|\left(\xi_1, \xi_2\right)\right|^{-|\alpha|},\left(\xi_1, \xi_2\right) \neq(0,0), \text { for }|\alpha| \leq 2 d.\]
	Then the multiplier operator $T_m$ defined as
	\[T_m\left(f_1, f_2\right)(x)=\int_{\mathbb{R}^{2 d}} m\left(\xi_1, \xi_2\right) \widehat{f}_1\left(\xi_1\right) \widehat{f}_2\left(\xi_2\right) e^{2 \pi i x \cdot\left(\xi_1+\xi_2\right)} d \xi_1 d \xi_2\]
	satisfies the following estimate
	\[\left\|T_m\right\|_{L^{p_1}\left(\mathbb{R}^d\right) \times L^{p_2}\left(\mathbb{R}^d\right) \rightarrow L^p\left(\mathbb{R}^d\right)} \lesssim C.\]
\end{theorem}
\subsection*{Estimates for Medium and High frequency terms:} We use the following strategy to obtain desired bounds for the operators corresponding to medium and high frequency terms. 
\begin{enumerate}
	\item For H\"older exponents $(p_1,p_2,p)$ in the Banach range, we prove a $(p_1,p_2,p)-$bound (see \Cref{lemma:p1p2pGrowth}) with an admissible growth in $k\in\N$ by employing the $L^p-$estimates of shifted square functions. 
	\item This growth estimate is further upgraded to a weak type bound at endpoints (see \Cref{lemma:endpointCZ}) lying in the non-Banach region by employing the Calder\'on-Zygmund analysis. 
	\item The crucial decay estimates (see \Cref{lemma:221}) for the operators $T_{\Omega,k}^{HH},\;T_{\Omega,k}^{HL}$, and $T_{\Omega,k}^{LH}$ at the points $(2,2,1),\;(2,\infty,2)$, and $(\infty,2,2)$ respectively are obtained by using our main result \Cref{thm:Smoothing}. 
	\item Finally, we use an interpolation argument between the estimates proved in  \Cref{lemma:p1p2pGrowth}, \Cref{lemma:endpointCZ} and \Cref{lemma:221} to  prove a decay estimate for the exponents $(p_1,p_2,p)$ in the required range. With this decay estimate we  conclude \Cref{Thm:Rough} by summing in $k\in\N$. 
\end{enumerate}
We now state the results described as above. 
\begin{lemma}\label{lemma:p1p2pGrowth}
	For all $k\in\N$ and $1\leq p_1,p_2,p<\infty$ with $\frac{1}{p_1}+\frac{1}{p_2}=\frac{1}{p}$, we have 
	\begin{align*}
		\|T_{\Omega,k}^{HH}(f_1,f_2)\|_p\lesssim&\; |k|^{\left|\frac{1}{p_1}-\frac{1}{2}\right|+\left|\frac{1}{p_2}-\frac{1}{2}\right|}\|\Omega\|_1\|f_1\|_{p_1}\|f_2\|_{p_2},\\
		\|T_{\Omega,k}^{HL}(f_1,f_2)\|_p\lesssim&\; |k|^{\left|\frac{1}{p_1}-\frac{1}{2}\right|+\left|\frac{1}{p'}-\frac{1}{2}\right|}\|\Omega\|_1\|f_1\|_{p_1}\|f_2\|_{p_2},\\
		\|T_{\Omega,k}^{LH}(f_1,f_2)\|_p\lesssim&\; |k|^{\left|\frac{1}{p'}-\frac{1}{2}\right|+\left|\frac{1}{p_2}-\frac{1}{2}\right|}\|\Omega\|_1\|f_1\|_{p_1}\|f_2\|_{p_2}.
	\end{align*}
\end{lemma}

\begin{lemma}\label{lemma:endpointCZ}
	For all $k\in\N$, $1< q\leq\infty$, $\frac{1}{2}\leq p<1$ with $\frac{1}{p}+\frac{1}{q}=2$, there exists $c>0$ (depending on $q,p$) such that for $I\in\{HL,LH,HH\}$, we have
	\begin{align*}
		\|T_{\Omega,k}^{I}(f_1,f_2)\|_{L^{p,\infty}(\R^d)}&\lesssim |k|^c\|\Omega\|_q\|f_1\|_{{L^{1}(\R^d)}}\|f_2\|_{{L^{\frac{p}{1-p}}(\R^d)}},\\
		\|T_{\Omega,k}^{I}(f_1,f_2)\|_{L^{p,\infty}(\R^d)}&\lesssim |k|^c\|\Omega\|_q\|f_1\|_{{L^{\frac{p}{1-p}}(\R^d)}}\|f_2\|_{{L^{1}(\R^d)}}.
	\end{align*}
\end{lemma}

\begin{lemma}\label{lemma:221}
	For all $k\in\N$, $1<q\leq\infty$, there exists $c>0$ such that we have 
	\begin{align*}
		\|T_{\Omega,k}^{HH}(f_1,f_2)\|_1&\lesssim 2^{-ck}\|\Omega\|_q\|f_1\|_2\|f_2\|_2,\\
		\|T_{\Omega,k}^{LH}(f_1,f_2)\|_2&\lesssim 2^{-ck}\|\Omega\|_q\|f_1\|_\infty\|f_2\|_2, \\
		\|T_{\Omega,k}^{HL}(f_1,f_2)\|_2&\lesssim 2^{-ck}\|\Omega\|_q\|f_1\|_2\|f_2\|_\infty.
	\end{align*}
\end{lemma}

\begin{lemma}\label{lemma:p1p2pDecay}
	For all $k\in\N$, $1<q\leq\infty$, $1< p_1,p_2<\infty$, $\frac{1}{2}<p<\infty$ with $\frac{1}{p_1}+\frac{1}{p_2}=\frac{1}{p}$ and $\frac{1}{p}+\frac{1}{q}<2$, there exists $c>0$ (depending on $q,p_1,p_2,p$) such that  
	\[\|T_{\Omega,k}^{I}(f_1,f_2)\|_p\lesssim 2^{-ck}\|\Omega\|_q\|f_1\|_{p_1}\|f_2\|_{p_2},\quad\text{where}\;I\in\{HL,LH,HH\}.\]
\end{lemma}
\subsection{Proof of \Cref{lemma:p1p2pGrowth}}
The proof relies on the sharp $L^p-$bounds for the shifted square function. For $\sigma\in\R^d$, and $i\in\Z$, we set $\psi_i^\sigma(x)=2^{id}\psi(2^ix-\sigma)$. Consider the shifted square function defined by 
\begin{equation*}
		S_{\psi}^{\sigma}f(x)=\left(\sum_{i\in\Z}|f*\psi_i^\sigma(x)|^2\right)^\frac{1}{2}.
	\end{equation*}
	The following sharp $L^r-$estimates for $1<r<\infty$ with optimal growth on the shift $\sigma$ for $S_{\psi}^{\sigma}f$ hold
	\begin{align}
		\|S_{\psi}^{\sigma}f\|_r\lesssim&\;\log(e+|\sigma|)^{\left|\frac{1}{r}-\frac{1}{2}\right|}\|f\|_r,\label{shiftsquare}
	\end{align}
	Further, we require the $L^r-$estimates for the shifted maximal function. 
\begin{align}
	\|\sup\limits_{i\in\Z}|f*\phi_i(x-2^{-i}\sigma)|\|_r\lesssim&\;\log(e+|\sigma|)^\frac{1}{r}\|f\|_r.\label{shiftmaximal}
\end{align}
The boundedness for the shifted square function was studied in \cite{Muscalu} and \cite{SteinHarmonicAnalysisRealVariableMethodOrtho} but the sharp dependence was obtained in \cite{Park2024}.
We refer to \cite[Theorems 1.6 and 1.3]{Park2024} for the proofs of sharp dependence on the shift for the shifted square and maximal function estimates \eqref{shiftsquare} and \eqref{shiftmaximal} respectively. 

We now prove the required estimate for the high frequency term $T_{\Omega,k}^{HH}$. By a change of variable argument followed by Cauchy-Schwarz, Minkowski's and H\"older's inequalities, we have that 
	\begin{align*}
		&\;\|T_{\Omega,k}^{HH}(f_1,f_2)\|_p\\
		\lesssim&\;\sum_{a=-10}^{10}\Bigg\|\int_{\R^{2d}}\Omega\left(\frac{(y_1,y_2)}{|(y_1,y_2)|}\right)|(y_1,y_2)|^{-2d}\beta(|(y_1,y_2)|)\\
		&\hspace{3cm}\times\sum_{i\in\Z}f_1*\psi_{i+k+a}(x-2^{-i}y_1)f_2*\psi_{i+k}(x-2^{-i}y_2)dy_1dy_2\Bigg\|_p\\
		\lesssim&\;\sum_{a=-10}^{10}\int_{\R^{2d}}\left|\Omega\left(\frac{(y_1,y_2)}{|(y_1,y_2)|}\right)\right||(y_1,y_2)|^{-2d}|\beta(|(y_1,y_2)|)|\\
		&\hspace{3cm}\left\|\left(\sum_{i\in\Z}|f_1*\psi_{i+k+a}(x-2^{-i}y_1)|^2\right)^\frac{1}{2}\right\|_{p_1}\left\|\left(\sum_{i\in\Z}|f_2*\psi_{i+k}(x-2^{-i}y_2)|^2\right)^\frac{1}{2}\right\|_{p_2}dy_1dy_2\\
		\lesssim&\;\sum_{a=-10}^{10}\int_{\R^{2d}}\left|\Omega\left(\frac{(y_1,y_2)}{|(y_1,y_2)|}\right)\right||(y_1,y_2)|^{-2d}|\beta(|(y_1,y_2)|)|\left\|S_{\psi_a}^{2^ky_1}f_1\right\|_{p_1}\left\|S_{\psi}^{2^ky_2}f_2\right\|_{p_2}dy_1dy_2\\
		\lesssim&\;\|f_1\|_{p_1}\|f_2\|_{p_2}\int_{\R^{2d}}\left|\Omega\left(\frac{(y_1,y_2)}{|(y_1,y_2)|}\right)\right||(y_1,y_2)|^{-2d}|\beta(|(y_1,y_2)|)|\log(e+|2^ky_1|)^{\left|\frac{1}{p_1}-\frac{1}{2}\right|}\\
		&\hspace{4cm}\times\log(e+|2^ky_2|)^{\left|\frac{1}{p_2}-\frac{1}{2}\right|}dy_1dy_2\\
		\lesssim&\; |k|^{\left|\frac{1}{p_1}-\frac{1}{2}\right|+\left|\frac{1}{p_2}-\frac{1}{2}\right|}\|\Omega\|_1\|f_1\|_{p_1}\|f_2\|_{p_2},
	\end{align*}
	where we used the estimate \eqref{shiftsquare} in the second last step.

	Next, we prove estimates for mixed frequency terms $T_{\Omega,k}^{HL}$ and $T_{\Omega,k}^{LH}$. Since the arguments for both the estimates are similar,  we indicate the proof for $T_{\Omega,k}^{HL}$ only. We observe that for $2^{i+k-1}\leq|\xi_1|\leq2^{i+k+1}$ and $|\xi_2|\leq2^{i+k-9}$ we have $2^{i+k-2}\leq|\xi_1+\xi_2|\leq2^{i+k+2}$. Hence we can write
	\[\langle T_\Omega^i(f_1*\psi_{i+k},f_2*\phi_{i+k-10}),f_3\rangle=\langle T_\Omega^i(f_1*\psi_{i+k},f_2*\phi_{i+k-10}),f_3*\widetilde{\psi}_{i+k}\rangle,\quad i\in\Z,\]
	where $\widetilde{\psi}\in\mathcal{S}(\R^d)$ is such that $\widehat{\widetilde{\psi}}(\xi)=1$ for $2^{-2}\leq|\xi|\leq2^2$ and $\widehat{\widetilde{\psi}}$ is supported in the annulus $\{2^{-3}\leq|\xi|\leq2^3\}$.

	By a change of variables argument followed by an application of Cauchy-Schwarz and H\"older's inequalities as before, we obtain
	\begin{align*}
		&\;|\langle T_{\Omega,k}^{HL}(f_1,f_2),f_3\rangle|\\
		\lesssim&\;\int_{\R^d}\int_{\R^{2d}}\left|\Omega\left(\frac{(y_1,y_2)}{|(y_1,y_2)|}\right)\right||(y_1,y_2)|^{-2d}|\beta(|(y_1,y_2)|)|\\
		&\;\hspace{2cm}\times\sum_{i\in\Z}|f_1*\psi_{i+k}(x-2^{-i}y_1)f_3*\widetilde{\psi}_{i+k}(x)f_2*\phi_{i+k-10}(x-2^{-i}y_2)|dy_1dy_2dx\\
		\lesssim&\;\int_{\R^d}\int_{\R^{2d}}\left|\Omega\left(\frac{(y_1,y_2)}{|(y_1,y_2)|}\right)\right||(y_1,y_2)|^{-2d}|\beta(|(y_1,y_2)|)|\\
		&\;\hspace{2cm}\times\sum_{i\in\Z}|f_1*\psi_{i+k}(x+2^{-i}y_2-2^{-i}y_1)f_3*\widetilde{\psi}_{i+k}(x+2^{-i}y_2)f_2*\phi_{i+k-10}(x)|dy_1dy_2dx\\
		\lesssim&\;\int_{\R^{2d}}\left|\Omega\left(\frac{(y_1,y_2)}{|(y_1,y_2)|}\right)\right||(y_1,y_2)|^{-2d}|\beta(|(y_1,y_2)|)|\left\|S_{\psi}^{2^k(y_1-y_2)}f_1\right\|_{p_1}\left\|S_{\widetilde{\psi}}^{-2^ky_2}f_3\right\|_{p'}\left\|Mf_2\right\|_{p_2}dy_1dy_2\\
		\lesssim&\;\|f_1\|_{p_1}\|f_2\|_{p_2}\|f_3\|_{p'}\int_{\R^{2d}}\left|\Omega\left(\frac{(y_1,y_2)}{|(y_1,y_2)|}\right)\right||(y_1,y_2)|^{-2d}|\beta(|(y_1,y_2)|)|\\
		&\;\hspace{6cm}\log(e+|2^k(y_1-y_2)|)^{\left|\frac{1}{p_1}-\frac{1}{2}\right|}\log(e+|2^ky_2|)^{\left|\frac{1}{p'}-\frac{1}{2}\right|}dy_1dy_2\\
		\lesssim&\; |k|^{\left|\frac{1}{p_1}-\frac{1}{2}\right|+{\left|\frac{1}{p'}-\frac{1}{2}\right|}}\|\Omega\|_1\|f_1\|_{p_1}\|f_2\|_{p_2}\|f_3\|_{p'},
	\end{align*}
	where we used the estimate \eqref{shiftsquare} and the $L^r-$ boundedness of the Hardy-Littlewood maximal function $M$ for $1<r<\infty$ in the second last step. This completes the proof of \Cref{lemma:p1p2pGrowth}.
\qed

\subsection{Proof of \Cref{lemma:endpointCZ}}
	The proof is similar for each term $I\in\{HL,LH,HH\}$ because it does not rely on the Fourier support of the functions $\psi$ and $\phi$. Therefore, we prove the estimate only for $I=HH$. We note that the required two inequalities collapse to a single $(1,1,\frac{1}{2})-$weak type estimate when $q=\infty$. We prove these two cases separately.
	\subsection*{Case 1. $\mathbf{1<q<\infty}$:} We prove the weak type estimate at the point $\left(1,\frac{p}{1-p},p\right)$. The estimate for the other point follows by symmetry. Without loss of generality, we may assume that  $\|f_1\|_{1}=\|f_2\|_{\frac{p}{1-p}}=\|\Omega\|_q=1$. For $\lambda>0$, we employ the Calder\'on-Zygmund decomposition to the function $f_1$ with height $\lambda^{p}$ to obtain a collection $\mathcal{C}$ of disjoint dyadic cubes in $\R^d$ such that
	\begin{enumerate}[i)]
		\item $f_1=g+b$ with $b=\sum_{s\in\Z}\left(\sum_{Q\in\mathcal{C}:|Q|=2^{-sd}}b_{Q}\right)=\sum_{s\in\Z}B^{s}$,
		\item $\|g\|_\infty\lesssim\lambda^p$ and $\|g\|_{1}\leq1$,
		\item $\int b_{Q}(x)dx=0$ and $\supp{b_{Q}}\subseteq Q$,
		\item $\|b_{Q}\|_{1}\lesssim\lambda^{p}|Q|$ and $\sum_{Q\in\mathcal{C}}|Q|\lesssim\frac{1}{\lambda^{p}}$. In particular, this implies $\sum_s\|B^{s}\|_{1}\lesssim1$.
	\end{enumerate}
	We define the exceptional set $E=\cup_{Q\in\mathcal{C}}4Q$. Thus $|E|\lesssim\frac{1}{\lambda^{p}}$. Consider
	\begin{align*}
		&\left|\left\{x\in \R^d:|T^{HH}_{\Omega,k}(f_1,f_2)(x)|>\lambda\right\}\right|\\
		\lesssim&\;|E|+\left|\left\{x\in \R^d:|T^{HH}_{\Omega,k}(g,f_2)(x)|>\frac{\lambda}{2}\right\}\right|+\left|\left\{x\in E^c:|T^{HH}_{\Omega,k}(b,f_2)(x)|>\frac{\lambda}{2}\right\}\right|\\
		=&\;|E|+I_g+I_b.
	\end{align*}
	\subsection*{Estimate for $I_g$:} To estimate the level set with good terms, we use the $L^{(2p)'}\times L^{\frac{p}{1-p}} \to L^{2p}-$boundedness of $T^{HH}_{\Omega,k}$ to obtain
	\begin{align*}
		I_g\lesssim&\;\frac{|k|^c}{\lambda^{2p}}\|g\|_{(2p)'}^{2p}\|f_2\|_{\frac{p}{1-p}}^{2p}\\
		\lesssim&\;\frac{|k|^c}{\lambda^{2p}}(\lambda^p)^{\left(1-\frac{1}{(2p)'}\right)2p}\|g\|_{1}^{2p}\|f_2\|_{\frac{p}{1-p}}^{2p}\\
		\lesssim&\;\frac{|k|^c}{\lambda^p}.
	\end{align*}
	\subsection*{Estimate for $I_b$:} To estimate $I_b$, it is enough to show that there exists a constant $c>0$ such that
	\[\left|\left\{x\in E^c:|\sum_{i\in\Z}T^i_\Omega(b*\psi_{i+k+a},f_2*\psi_{i+k})(x)|>\frac{\lambda}{2}\right\}\right|\lesssim \frac{1}{\lambda^p},\]
	for all $a\in\{-10,\dots,10\}$. Without loss of generality, we will prove this for $a=0$ only. We observe that for $x\in E^c$, $x$ lies at least $2^{-s+1}$ distance away from every cube $Q\in\mathcal{C}$ of measure $2^{-sd}$. Thus, for $s<i$, the distance of $x+2^{-i}y_2$ from the cube $Q$ is larger than $2^{-s}$ for any $(y_1,y_2)\in\supp{K^0}\subset\{\frac{1}{2}\leq|(y_1,y_2)|\leq2\}$. This implies that
	\[T_\Omega^i\left(B^{s}*\psi_{k+i},f_2*\psi_{k+i}\right)(x)=T_\Omega^i\left(B^{s}*(\chi_{|\cdot|>2^{-s}}\psi_{k+i}),f_2*\psi_{k+i}\right)(x),\;\text{for}\;s<i\;\text{and}\;x\in E^c.\]
	We set $\psi_{k+i}^s=\begin{cases}
	\psi_{k+i}, &s\geq i\\ \chi_{|\cdot|>2^{-s}}\psi_{k+i}, &s<i.
	\end{cases}$. 
	
	Next, we have the following estimates for their interactions with the bad functions $B_s$.
	\begin{lemma}\label{lemma:badterm}
		For any $k\in\N$, we have
		\begin{enumerate}[i)]
			\item $\|B^{s}*\psi_{k+i}^s\|_{1}\lesssim\min\{1,2^{k+i-s}\}\|B^{s}\|_{1}$, for $s\geq i$.
			\item For any $s<i$ and $N\in\N$, $\|B^{s}*\psi_{k+i}^s\|_{1}\lesssim 2^{-(k+i-s)N}\|B^{s}\|_{1}$.
		\end{enumerate}
	\end{lemma}
	\begin{proof}
		For the first estimate, we note that \[\|B^{s}*\psi_{k+i}^s\|_{1}\lesssim\|B^{s}\|_{1}\|\psi_{k+i}^s\|_1\lesssim\|B^{s}\|_{1}.\] By the mean zero condition of $b_{Q}$ for each $Q\in\mathcal{C}$ and the mean value theorem, we have
		\[\|b_{Q}*\psi_{k+i}\|_1\lesssim 2^{k+i-s}\|b_{Q}\|_{1}.\]
		After summing in $Q\in\mathcal{C}$ we obtain
		\[\|B^{s}*\psi_{k+i}\|_{1}\lesssim2^{k+i-s}\|B^{s}\|_{1}.\]
		For the second estimate,  since $\psi\in\mathcal{S}(\R^d)$, we have
		\[\|\chi_{|\cdot|>2^{-s}}\psi_{k+i}\|_1\lesssim2^{-(k+i-s)N}.\]
		This produces the required estimate.
\end{proof}

	By a change of variable argument and H\"older's inequality with exponents  $\frac{1}{q}+\frac{1-p}{p}=1$, we have the following pointwise domination
	\begin{align*}
		&\;\left|\sum_{i\in\Z}T^i_\Omega(b*\psi_{i+k},f_2*\psi_{i+k})(x)\right|\\
		\lesssim&\;\sum_{i\in\Z}\int_{\R^{2d}}|K^0(y_1,y_2)|\sum_{s\in\Z} |B^s*\psi_{k+i}^s(x-2^{-i}y_1)||f_2*\psi_{k+i}(x-2^{-i}y_2)|dy_1dy_2\\
		\lesssim&\;\int_{\R^{2d}}|K^0(y_1,y_2)|\sum_{i\in\Z}\sum_{s\in\Z} |B^s*\psi_{k+i}^s(x-2^{-i}y_1)|Mf_2(x-2^{-i}y_2)dy_1dy_2\\
		\lesssim&\;\int_{|y_1|\leq2}\sum_{i\in\Z}\sum_{s\in\Z} |B^s*\psi_{k+i}^s(x-2^{-i}y_1)|\left(\int_{|y_2|\leq2}|K^0(y_1,y_2)|^qdy_2\right)^\frac{1}{q}\\
		&\hspace{3cm}\times\left(\int_{|y_2|\leq2}Mf_2(x-2^{-i}y_2)^{\frac{p}{1-p}}dy_2\right)^\frac{1-p}{p}dy_1\\
		\lesssim&\;M_{\frac{p}{1-p}}Mf_2(x)\left(\int_{|y_1|\leq2}\sum_{i\in\Z}\sum_{s\in\Z} |B^s*\psi_{k+i}^s(x-2^{-i}y_1)|\left(\int_{|y_2|\leq2}|K^0(y_1,y_2)|^qdy_2\right)^\frac{1}{q}dy_1\right).
	\end{align*}
	This yields
	\begin{align*}
		&\;\left|\left\{x\in E^c:|\sum_{i\in\Z}T^i_\Omega(b*\psi_{i+k+a},f_2*\psi_{i+k})(x)|>\frac{\lambda}{2}\right\}\right|\\
		\lesssim&\;\left|\left\{x\in E^c:M_{\frac{p}{1-p}}Mf_2(x)>\lambda^{1-p}\right\}\right|\\
		&+\left|\left\{x\in E^c:\left(\int_{|y_1|\leq2}\sum_{i\in\Z}\sum_{s\in\Z} |B^s*\psi_{k+i}^s(x-2^{-i}y_1)|\left(\int_{|y_2|\leq2}|K^0(y_1,y_2)|^qdy_2\right)^\frac{1}{q}dy_1\right)>\lambda^p\right\}\right|.
	\end{align*}
	The estimate for the first term follows by the weak type $(\frac{p}{1-p},\frac{p}{1-p})-$boundedness of $M_{\frac{p}{1-p}}$ and strong type boundedness of $M$. By Chebyshev's inequality and \Cref{lemma:badterm}, the second term is dominated by
	\begin{align*}
		&\;\frac{1}{\lambda^p}\int_{\R^d}\int_{|y_1|\leq2}\sum_{i\in\Z}\sum_{s\in\Z} |B^s*\psi_{k+i}^s(x-2^{-i}y_1)|\left(\int_{|y_2|\leq2}|K^0(y_1,y_2)|^qdy_2\right)^\frac{1}{q}dy_1dx\\
		\lesssim&\;\frac{1}{\lambda^p}\left(\int_{|y_1|\leq2}\left(\int_{|y_2|\leq2}|K^0(y_1,y_2)|^qdy_2\right)^\frac{1}{q}dy_1\right)\sum_{s\in\Z}\sum_{i\in\Z} \|B^s*\psi_{k+i}^s(x)\|_1\\
		\lesssim&\;\frac{1}{\lambda^p}\|\Omega\|_q\sum_{s\in\Z}\sum_{i\in\Z} \min\{2^{k+i-s},2^{-(k+i-s)}\}\|B^s\|_1\\
		\lesssim&\;\frac{1}{\lambda^p}.
	\end{align*}
	This concludes the proof for the case when $1<q<\infty$.

	\subsection*{Case 2. $\mathbf{q=\infty}$:} Note that in this case $p=\frac{1}{2}$. Let $f_1,f_2\in L^1(\R^d)$. Without loss of generality we may assume that $\|f_l\|_1=1,\;l=1,2$. In this case, we need to apply Calder\'on-Zygmund decomposition to both the functions $f_1,f_2$ with height $\lambda^\frac{1}{2}$. This gives us two collections $\mathcal{C}_1,\mathcal{C}_2$ of disjoint dyadic cubes such that for $l=1,2$, we have
	\begin{enumerate}[i)]
		\item $f_l=g_l+b_l$ with $b_l=\sum_{s_l\in\Z}\left(\sum_{Q\in\mathcal{C}_l:|Q|=2^{-s_ld}}b_{l,Q}\right)=\sum_{s\in\Z}B_l^{s_l}$,
		\item $\|g_l\|_\infty\lesssim\lambda^\frac{1}{2}$ and $\|g_l\|_{1}\leq1$,
		\item $\int b_{l,Q}(x)dx=0$ and $\supp{b_{l,Q}}\subseteq Q$,
		\item $\|b_{l,Q}\|_{1}\lesssim\lambda^{\frac{1}{2}}|Q|$ and $\sum_{Q\in\mathcal{C}_l}|Q|\lesssim\frac{1}{\lambda^{\frac{1}{2}}}$. In particular, this implies $\sum_{s_l}\|B_l^{s_l}\|_{1}\lesssim1$.
	\end{enumerate}
	As earlier, we define the exceptional set $E=\cup_{Q\in\mathcal{C}_1\cup\mathcal{C}_2}4Q$ and note that $|E|\lesssim\frac{1}{\lambda^{\frac{1}{2}}}$. Consider
	\begin{align*}
		&\left|\left\{x\in \R^d:|T^{HH}_{\Omega,k}(f_1,f_2)(x)|>\lambda\right\}\right|\\
		\lesssim&\;|E|+\left|\left\{x\in \R^d:|T^{HH}_{\Omega,k}(g_1,g_2)(x)|>\frac{\lambda}{4}\right\}\right|\\
		&+\left|\left\{x\in \R^d:|T^{HH}_{\Omega,k}(b_1,g_2)(x)|>\frac{\lambda}{4}\right\}\right|+\left|\left\{x\in \R^d:|T^{HH}_{\Omega,k}(g_1,b_2)(x)|>\frac{\lambda}{4}\right\}\right|\\
		&+\left|\left\{x\in \R^d:|T^{HH}_{\Omega,k}(b_1,b_2)(x)|>\frac{\lambda}{4}\right\}\right|\\
		=&\;|E|+I_{gg}+I_{bg}+I_{gb}+I_{bb}.
	\end{align*}
	\subsection*{Estimate for $I_{gg}$:} The desired estimate for $I_{gg}$ follows from $L^2\times L^2\to L^1-$boundedness of $T_{\Omega,k}^{HH}$ from \Cref{lemma:p1p2pGrowth}. 
	\subsection*{Estimate for $I_{gb}$ and $I_{bg}$:} By symmetry, it is enough to bound $I_{gb}$. By Chebyshev's inequality and \Cref{lemma:badterm}, we have
	\begin{align*}
		I_{gb}\lesssim&\;\frac{1}{\lambda}\sum_{i\in\Z}\sum_{s_1\in\Z}\int_{\R^d}\int |K^i(y_1,y_2)||B_1^{s_1}*\psi_{k+i}^{s_1}(x-y_1)||g_2*\psi_{k+i}(x-y_2)|dy_1dy_2dx\\
		\lesssim&\;\frac{1}{\lambda}\sum_{i\in\Z}\sum_{s_1\in\Z}\|K^i\|_1\|B_1^{s_1}*\psi_{k+i}^{s_1}\|_1\|g_2*\psi_{k+i}\|_\infty\\
		\lesssim&\;\frac{1}{\lambda^\frac{1}{2}}\|\Omega\|_1\sum_{i\in\Z}\sum_{s_1\in\Z}\min\{2^{k+i-s_1},2^{-(k+i-s_1)}\}\|B^{s_1}_1\|_1\\
		\lesssim&\;\frac{1}{\lambda^\frac{1}{2}}.
	\end{align*}
	\subsection*{Estimate for $I_{bb}$:} We apply Chebyshev's inequality and use the single scale estimate \eqref{singlescale:Lpbound} and \Cref{lemma:badterm} to obtain
	\begin{align*}
		I_{bb}\lesssim&\;\frac{1}{\lambda^\frac{1}{2}}\sum_{i\in\Z}\sum_{s_1\in\Z}\sum_{s_2\in\Z}\|T^i_\Omega(B_1^{s_1}*\psi_{k+i}^{s_1},B_2^{s_2}*\psi_{k+i}^{s_2})\|_\frac{1}{2}^\frac{1}{2}\\
		\lesssim&\;\frac{1}{\lambda^\frac{1}{2}}\sum_{i\in\Z}\sum_{s_1\in\Z}\sum_{s_2\in\Z}\|B_1^{s_1}*\psi_{k+i}^{s_1}\|_1^\frac{1}{2}\|B_2^{s_2}*\psi_{k+i}^{s_2}\|_1^\frac{1}{2}\\
		\lesssim&\;\frac{1}{\lambda^\frac{1}{2}}\sum_{s_1\in\Z}\sum_{s_2\in\Z}\sum_{i\in\Z}\min_{l=1,2}\min\{1,2^{\frac{k+i-s_l}{2}},2^{-\frac{(i-s_l)}{2}}\}\|B_1^{s_1}\|_1^\frac{1}{2}\|B_2^{s_2}\|_1^\frac{1}{2}\\
		\lesssim&\;\frac{1}{\lambda^\frac{1}{2}}|k|\sum_{s_1\in\Z}\sum_{s_2\in\Z}\min\{1,2^{-\frac{(|s_1-s_2|-|k|)}{2}}\}\|B_1^{s_1}\|_1^\frac{1}{2}\|B_2^{s_2}\|_1^\frac{1}{2}\\
		\lesssim&\;\frac{1}{\lambda^\frac{1}{2}}|k|\left(\sum_{s_1\in\Z}\sum_{s_2\in\Z}\min\{1,2^{-\frac{(|s_1-s_2|-|k|)}{2}}\}\|B_1^{s_1}\|_1\right)^\frac{1}{2}\left(\sum_{s_1\in\Z}\sum_{s_2\in\Z}\min\{1,2^{-\frac{(|s_1-s_2|-|k|)}{2}}\}\|B_2^{s_2}\|_1\right)^\frac{1}{2}\\
		\lesssim&\;\frac{1}{\lambda^\frac{1}{2}}|k|^2.
	\end{align*}
	This concludes the proof of \Cref{lemma:endpointCZ}.
\qed

\subsection{Proof of \Cref{lemma:221}}
	For the high frequency term $I=HH$, we use \Cref{thm:Smoothing} and Cauchy-Schwarz inequality along with bounded overlapping property of Fourier supports of functions $\widetilde{\psi}_{i+k},i\in\Z$ to obtain
	\begin{align*}
		\|T_{\Omega,k}^{HH}(f_1,f_2)\|_1\lesssim&\;\sum_{i\in\Z}\sum_{a=-10}^{10}\|T^i_\Omega(f_1*\widetilde{\psi}_{i+k+a}*\psi_{i+k+a},f_2*\widetilde{\psi}_{i+k}*\psi_{i+k})\|_1\\
		\lesssim&\;2^{-ck}\|\Omega\|_q\sum_{i\in\Z}\sum_{a=-10}^{10}\|f_1*\widetilde{\psi}_{i+k+a}\|_2\|f_2*\widetilde{\psi}_{i+k}\|_2\\
		\lesssim&\;2^{-ck}\|\Omega\|_q\sum_{a=-10}^{10}\;\left(\sum_{i\in\Z}\|f_1*\widetilde{\psi}_{i+k+a}\|_2^2\right)^\frac{1}{2}\left(\sum_{i\in\Z}\|f_2*\widetilde{\psi}_{i+k}\|_2^2\right)^\frac{1}{2}\\
		\lesssim&\;2^{-ck}\|\Omega\|_q\|f_1\|_2\|f_2\|_2.
	\end{align*}
	As previously, since the proof of desired estimates for mixed frequency terms $I=LH,HL$ are similar in nature. We provide details only for $T_{\Omega,k}^{LH}$. 
	
	First observe that Fourier support of $T_{\Omega}^i(f_1*\phi_{i+k-10},f_2*\psi_{i+k}),\;i\in\Z$ lies in the annular region  $\{2^{i+k-4} \leq|\xi| \leq 2^{i+k+4}\}$. Therefore, by applying the estimate \eqref{inq:Trilinearsmoothing1} from \Cref{thm:Smoothing}, we have that
	\begin{align*}
		\sum_{i\in\Z}\|T_{\Omega}^i(f_1*\phi_{i+k-10},f_2*\psi_{i+k})\|_2^2\lesssim&\;2^{-2ck}\|\Omega\|_q^2\sum_{i\in\Z}\|f_1*\phi_{i+k-10}\|_\infty^2\|f_2*\psi_{i+k}\|_2^2\\
		\lesssim&\;2^{-2ck}\|\Omega\|_q^2\|f_1\|_\infty^2\|f_2\|_2^2.
	\end{align*}
	Next, we apply \Cref{lemma:GrafMedium} to get that 
	\begin{align*}
		\|T_{\Omega,k}^{LH}(f_1,f_2)\|_2\lesssim&\;\left\|\left(\sum_{i \in \Z}\left|T_{\Omega}^i(f_1*\phi_{i+k-10},f_2*\psi_{i+k})\right|^2\right)^{\frac{1}{2}}\right\|_{2}\\
		=&\;\left(\sum_{i\in\Z}\|T_{\Omega}^i(f_1*\phi_{i+k-10},f_2*\psi_{i+k})\|_2^2\right)^\frac{1}{2}\\
		\lesssim&\;2^{-ck}\|\Omega\|_q\|f_1\|_\infty\|f_2\|_2.
	\end{align*}
\qed
\section{Proof of \Cref{Thm:RoughBanach}}\label{sec:proofrough1}
Without loss of generality we assume $\|\Omega\|_{L(\log L)^A}=1$. Also,  fix the exponents $p_1,p_2,p$ throughout the proof. We follow exactly the same decomposition as that in the proof of \Cref{Thm:Rough}. We have
\[T_\Omega(f_1,f_2)(x)=T_{\Omega}^{LL}(f_1,f_2)(x)+\sum_{k\in\N}T_{\Omega,k}^{HL}(f_1,f_2)(x)+\sum_{k\in\N}T_{\Omega,k}^{LH}(f_1,f_2)(x)+\sum_{k\in\N}T_{\Omega,k}^{HH}(f_1,f_2)(x).\]
As earlier, $T_\Omega^{LL}$ is bounded in the required range for all $\Omega\in L^1(\Sp^{2d-1})$. The desired estimates for the remaining terms are obtained by exploiting the exponential decay in $k$ in \Cref{lemma:p1p2pDecay} along with a decomposition of the function $\Omega$ into appropriate $L^\infty$ and $L^1$ parts. For $I\in\{HL,LH,HH\}$ we write 
\[T_{\Omega,k}^I=T_{\Omega^+,k}^I+T_{\Omega^-,k}^I, \]
where 
\[\Omega^+(\theta)=\Omega(\theta)\chi_{\{|\Omega(\cdot)|\leq2^{\frac{ck}{2}}\}}(\theta)\quad\text{and}\quad \Omega^-(\theta)=\Omega(\theta)\chi_{\{|\Omega(\cdot)|>2^{\frac{ck}{2}}\}}(\theta),\]
with $c>0$ is as in \Cref{lemma:p1p2pDecay}. Applying \Cref{lemma:p1p2pDecay}, we have
\begin{align*}
	\left\|\sum_{k\in\N}T_{\Omega^+,k}^{I}(f_1,f_2)\right\|_p\lesssim&\;\sum_{k\in\N}2^{-ck}\|\Omega^+\|_q\|f_1\|_{p_1}\|f_2\|_{p_2}\\
	\lesssim&\;\sum_{k\in\N}2^{-\frac{ck}{2}}\|f_1\|_{p_1}\|f_2\|_{p_2}\lesssim\|f_1\|_{p_1}\|f_2\|_{p_2}.
\end{align*}
We set 
\[a_{HH}=\left|\frac{1}{p_1}-\frac{1}{2}\right|+\left|\frac{1}{p_2}-\frac{1}{2}\right|,\quad a_{HL}=\left|\frac{1}{p_1}-\frac{1}{2}\right|+\left|\frac{1}{p'}-\frac{1}{2}\right|,\quad a_{LH}=\left|\frac{1}{p_2}-\frac{1}{2}\right|+\left|\frac{1}{p'}-\frac{1}{2}\right|.\]
To obtain the estimate for $T_{\Omega^-,k}^{I},\;I\in\{HL,LH,HH\}$, we use \Cref{lemma:p1p2pGrowth} to get
\begin{align*}
	\left\|\sum_{k\in\N}T_{\Omega^-,k}^{I}(f_1,f_2)\right\|_p\lesssim&\;\sum_{k\in\N}k^{a_I}\|\Omega^-\|_1\|f_1\|_{p_1}\|f_2\|_{p_2}\\
	\lesssim&\;\|f_1\|_{p_1}\|f_2\|_{p_2}\int_{\Sp^{2d-1}}|\Omega(\theta)|\sum_{k\in\N}k^{a_I}\chi_{|\Omega(\cdot)|>2^{\frac{ck}{2}}}(\theta)d\sigma(\theta)\\
	\lesssim&\;\|f_1\|_{p_1}\|f_2\|_{p_2}\int_{\Sp^{2d-1}}|\Omega(\theta)|\log(e+|\Omega(\theta)|)^{1+{a_I}}d\sigma(\theta)\\
	\leq&\;\|f_1\|_{p_1}\|f_2\|_{p_2}.
\end{align*}
This concludes the proof of \Cref{Thm:RoughBanach}.
\section{Proof of \Cref{Thm:Roughsinglemaximal}}\label{sec:proofsinglemaximal}
We note that it is enough to estimate the operator \[\sup_{i\in\Z}|K^i*(f_1\otimes f_2)(x,x)|=\sup_{i\in\Z} |T_\Omega^i(f_1,f_2)(x)|,\] where $K^i(y_1,y_2)=\Omega\left(\frac{(y_1,y_2)}{|(y_1,y_2)|}\right)2^{2di}\beta_i(|(y_1,y_2)|)$ with $\Omega\geq0$ and $\beta$ is a compactly supported non-negative radial function which takes the value $1$ on $[\frac{1}{2},2]$. Using the notation as in proof of \Cref{Thm:Rough}, we write
\begin{align*}
	&\sup_{i\in\Z}|T_\Omega^i(f_1,f_2)(x)|\\
	\lesssim&\;\sup_{i\in\Z}\Big|T_\Omega^i(f_1*\phi_i,f_2*\phi_i)(x)\Big|+\sum_{k\in\N}\sup\limits_{i\in\Z}\Big|T^i_\Omega(f_1*\psi_{i+k},f_2*\phi_{i+k-10})(x)\Big|,\\
	&\;+\sum_{k\in\N}\sup\limits_{i\in\Z}\Big|T^i_\Omega(f_1*\phi_{i+k-10},f_2*\psi_{i+k})(x)\Big|+
	\sum_{a=-10}^{10}\sum_{k\in\N}
	\sup\limits_{i\in\Z}\Big|T^i_\Omega(f_1*\psi_{i+k+a},f_2*\psi_{i+k})(x)\Big|.\\
	&\;=:M^{LL}_{\Omega}(f_1,f_2)(x)+\sum_{k\in\N}M^{HL}_{\Omega,k}(f_1,f_2)(x)+\sum_{k\in\N}M^{LH}_{\Omega,k}(f_1,f_2)(x)+\sum_{k\in\N}M^{HH}_{\Omega,k}(f_1,f_2)(x).
\end{align*}
\subsection*{Estimate for $M^{LL}_{\Omega}$:}The estimate for the first term on the right side of the expression above follows from the pointwise bound,
\[M^{LL}_{\Omega}(f_1,f_2)(x)\lesssim\|\Omega\|_1Mf_1(x)Mf_2(x),\]
which is an easy consequence of the bound $|f*\phi_i(x-2^{-i}y)|\lesssim Mf(x)$, for any $y\in B(0,2)$.

\subsection*{Estimates for $M^{HL}_{\Omega,k}$, $M^{LH}_{\Omega,k}$, and $M^{HH}_{\Omega,k}$:} We note that by the scheme used in the proof of \Cref{Thm:Rough}, it is enough to prove the following lemma.
\begin{lemma}\label{lemma:MOmega}
	Let $k\in\N$. Then, the following estimstes hold true.
\begin{enumerate}
	\item For $\frac{1}{p}=\frac{1}{p_1}+\frac{1}{p_2}\leq1$ and $I\in\{HL,LH,HH\}$, there exists $c>0$ such that
	\begin{align}\label{est:MOmegagrowth}
		\|M^{I}_{\Omega,k}(f_1,f_2)\|_p&\lesssim k^c\|\Omega\|_1\|f_1\|_{p_1}\|f_2\|_{p_2}.
	\end{align}
	\item For $1< q\leq\infty$, there exists $c>0$ such that for $I\in\{HL,LH,HH\}$, we have
	\begin{align}\label{est:MOmegaendpoint}
		\|M_{\Omega,k}^{I}(f_1,f_2)\|_{L^{\frac{1}{2},\infty}}&\lesssim k^c\|\Omega\|_q\|f_1\|_{1}\|f_2\|_{1},
	\end{align} 
	for all $\Omega\in L^q(\Sp^1)$ with $\supp{\Omega}\subset\{(\theta_1,\theta_2)\in\Sp^1:\;|\theta_1-\theta_2|\geq\frac{1}{2}\}$.
	\item For $1<q\leq\infty$, there exists $C>0$ so that 
	\begin{equation}\label{est:MOmegadecay}
		\begin{aligned}
			\|M^{HL}_{\Omega,k}(f_1,f_2)\|_2&\lesssim 2^{-Ck}\|\Omega\|_q\|f_1\|_2\|f_2\|_\infty,\\
			\|M^{LH}_{\Omega,k}(f_1,f_2)\|_2&\lesssim 2^{-Ck}\|\Omega\|_q\|f_1\|_\infty\|f_2\|_2,\\
			\|M^{HH}_{\Omega,k}(f_1,f_2)\|_1&\lesssim 2^{-Ck}\|\Omega\|_q\|f_1\|_2\|f_2\|_2.
		\end{aligned}
	\end{equation}
\end{enumerate}
\end{lemma}
\begin{proof}
	For convenience we set,
	\begin{equation}\label{notationforpieces}
		\begin{aligned}
		T_{\Omega,k}^{HL,i}(f_1,f_2)(x)&=T^i_\Omega(f_1*\psi_{i+k},f_2*\phi_{i+k-10})(x),\\
		T_{\Omega,k}^{LH,i}(f_1,f_2)(x)&=T^i_\Omega(f_1*\phi_{i+k-10},f_2*\psi_{i+k})(x),\\
		T_{\Omega,k}^{HH,i}(f_1,f_2)(x)&=\sum_{a=-10}^{10}\Big|T^i_\Omega(f_1*\psi_{i+k+a},f_2*\psi_{i+k})(x)\Big|
	\end{aligned}
	\end{equation}
	By the embeddings $\ell_1\subset\ell_2\subset\ell_\infty$, we have the following pointwise bounds
	\begin{align}
		|M_{\Omega,k}^{I}(f_1,f_2)(x)|&\leq\sum\limits_{i\in\Z}\Big|T^{I,i}_{\Omega,k}(f_1,f_2)(x)\Big|,\quad I\in\{HL,LH,HH\},\label{Highptwise}\\
		|M_{\Omega,k}^{I}(f_1,f_2)(x)|&\leq\left(\sum\limits_{i\in\Z}\Big|T^{I,i}_{\Omega,k}(f_1,f_2)(x)\Big|^2\right)^\frac{1}{2},\quad I\in\{HL,LH\}.\label{Mediumptwise}
	\end{align}
	Hence, the estimate \eqref{est:MOmegagrowth} follows from Minkowski's and H\"older's inequalities along with the bounds \eqref{shiftsquare} and \eqref{shiftmaximal} for the  shifted square function and shifted maximal function respectively.
	
	The proof of the weak type estimate \eqref{est:MOmegaendpoint} is similar to the proof of Case 2 in \Cref{lemma:endpointCZ}. We point out the main differences and leave the details to the readers. We use the $(2,2,1)-$ estimate \eqref{est:MOmegadecay} instead of \Cref{lemma:p1p2pGrowth} for the corresponding term $I_{gg}$ obtained after applying Calder\'on-Zygmund decomposition to the functions $f_1$ and $f_2$. The analysis of the corresponding terms $I_{gb},I_{bg}$ and $I_{bb}$ are similar once we use the pointwise bound \eqref{Highptwise}, with the only exception that we use the single scale estimate \eqref{singlescale:Lpbound_d=1} instead of \eqref{singlescale:L1bound} for the estimate of $I_{bb}$.

	The proof of the decay estimates \eqref{est:MOmegadecay} follows by using estimates \eqref{Mediumptwise} and \eqref{Highptwise} and proceeding exactly the same way as that in the proof of \Cref{lemma:221}.
\end{proof}

\section{Proof of \Cref{Thm:Roughmaximal}}\label{sec:proofmaximal}
First, we note that \[T_\Omega^*(f_1,f_2)(x)\leq\sup\limits_{j\in\Z}\Big|\sum\limits_{i>j}T_\Omega^i(f_1,f_2)(x)\Big|+M_{|\Omega|}(f_1,f_2)(x).\]
By \Cref{Thm:Roughsinglemaximal}, it remains to estimate $\sup\limits_{j\in\Z}\Big|\sum\limits_{i>j}T_\Omega^i(f_1,f_2)(x)\Big|$. We write
\begin{align*}
	\sup\limits_{j\in\Z}\Big|\sum\limits_{i>j}T_\Omega^i(f_1,f_2)(x)\Big|\lesssim&\;T_{\Omega}^{LL,*}(f_1,f_2)(x)+\sum_{k\in\N}T_{\Omega,k}^{HL,*}(f_1,f_2)(x)+\sum_{k\in\N}T_{\Omega,k}^{LH,*}(f_1,f_2)(x)\\
	&\;+\sum_{k\in\N}T_{\Omega,k}^{HH,*}(f_1,f_2)(x),
\end{align*}
where 
\begin{align*}
	T_{\Omega}^{LL,*}(f_1,f_2)(x)=&\;\sup\limits_{j\in\Z}\Big|\sum\limits_{i>j}T^i_\Omega(f_1*\phi_i,f_2*\phi_i)(x)\Big|,\\
	T_{\Omega,k}^{HL,*}(f_1,f_2)(x)=&\;\sup\limits_{j\in\Z}\Big|\sum\limits_{i>j}T^i_\Omega(f_1*\psi_{i+k},f_2*\phi_{i+k-10})(x)\Big|,\\
	T_{\Omega,k}^{LH,*}(f_1,f_2)(x)=&\;\sup\limits_{j\in\Z}\Big|\sum\limits_{i>j}T^i_\Omega(f_1*\phi_{i+k-10},f_2*\psi_{i+k})(x)\Big|,\\
	T_{\Omega,k}^{HH,*}(f_1,f_2)(x)=&\;\sum_{a=-10}^{10}\sup\limits_{j\in\Z}\Big|\sum\limits_{i>j}T^i_\Omega(f_1*\psi_{i+k+a},f_2*\psi_{i+k})(x)\Big|.
\end{align*}
\subsection*{Estimate for $T^{LL,*}_{\Omega}$:} Since the kernel $\sum_{i\in\Z}K^i*(\phi_i\otimes\phi_i)$ is a Calder\'on-Zygmund kernel, the maximal operator $T_{\Omega}^{LL,*}$ is bounded in the region $\mathcal{H}^\infty$ by Cotlar's inequality from \cite{GT2002}.

\subsection*{Estimate for $T^{I,*}_{\Omega,k},\;I\in\{HL,LH,HH\}$:} It is enough to show that \Cref{lemma:MOmega} also holds for $T^{I,*}_{\Omega,k}$ instead of $M^{I,*}_{\Omega,k}$ for $I\in\{HL,LH,HH\}$. Indeed, by employing the arguments used in the proof of \Cref{lemma:MOmega}, these estimates for $T^{I,*}_{\Omega,k}$ follow once we establish the following two inequalities.
\begin{align*}
	|T_{\Omega,k}^{I,*}(f_1,f_2)(x)|&\leq\sum\limits_{i\in\Z}\Big|T^{I,i}_\Omega(f_1,f_2)(x)\Big|,\quad I\in\{HL,LH,HH\},\\
	|T_{\Omega,k}^{I,*}(f_1,f_2)(x)|&\lesssim\left|T^{I}_{\Omega,k}(f_1,f_2)(x)\right|+M(T_{\Omega,k}^I(f_1,f_2))(x)+M(M_{\Omega,k}^I(f_1,f_2))(x),\quad I\in\{HL,LH\}.
\end{align*}
The first inequality is straightforward. Therefore, we provide the arguments for the second.  We follow the notation introduced in \eqref{notationforpieces} to write,
\begin{align*}
	\sup_{j\in\Z}\left|\sum_{i>j}T^{I,i}_{\Omega,k}(f_1,f_2)(x)\right|\leq \left|\sum_{i\in\Z}T^{I,i}_{\Omega,k}(f_1,f_2)(x)\right|+\sup_{j\in\Z}\left|\sum_{i\leq j}T^{I,i}_{\Omega,k}(f_1,f_2)(x)\right|.
\end{align*}
Thus, it is enough to estimate the second term on the right hand side in the expression above. The Fourier support property $\supp{(T^{I,i}_{\Omega,k}(f_1,f_2))^\wedge}\subset\{2^{i+k-2}\leq|\xi|\leq 2^{i+k+2}\}$ implies that 
\[\bigcup_{i\leq j}\supp{(T^{I,i}_{\Omega,k}(f_1,f_2))^\wedge}\subset\{|\xi|\leq 2^{j+k+2}\},\]
and thus, the second term can be rewritten as
\begin{align*}
	&\sup_{j\in\Z}\left|\sum_{i\leq j}T^{I,i}_{\Omega,k}(f_1,f_2)(x)\right|\\
	=&\sup_{j\in\Z}\left|\phi_{j+k+2}*\left(\sum_{i\leq j}T^{I,i}_{\Omega,k}(f_1,f_2)\right)(x)\right|\\
	\leq&\sup_{j\in\Z}\left|\phi_{j+k+2}*\left(\sum_{i\in\Z}T^{I,i}_{\Omega,k}(f_1,f_2)\right)(x)\right|+\sup_{j\in\Z}\left|\phi_{j+k+2}*\left(\sum_{i> j}T^{I,i}_{\Omega,k}(f_1,f_2)\right)(x)\right|\\
	\leq&\sup_{j\in\Z}\left|\phi_{j+k+2}*\left(\sum_{i\in\Z}T^{I,i}_{\Omega,k}(f_1,f_2)\right)(x)\right|+\sup_{j\in\Z}\left|\phi_{j+k+2}*\left(\sum_{j+5>i> j}T^{I,i}_{\Omega,k}(f_1,f_2)\right)(x)\right|\\
	\lesssim& M(T_{\Omega,k}^I(f_1,f_2))(x)+M(M_{\Omega,k}^I(f_1,f_2))(x).
\end{align*}
This concludes the proof of \Cref{Thm:Roughmaximal}.
\section*{Acknowledgement}
We would like to thank Lenka Slav\'ikov\'a and Fred Yu-Hsiang Lin for a careful reading of our manuscript and their valuable comments. We would also like to thank the anonymous reviewer for helpful comments and suggestions. Ankit Bhojak is supported by the Science and Engineering Research Board, Department of Science and Technology, Govt. of India, under the scheme National Post-Doctoral Fellowship, file no. PDF/2023/000708. Saurabh Shrivastava acknowledges the support by Anusandhan National Research Foundation, India under the project ANRF/ARG/2025/000940/MS. 
\bibliography{biblio}

\begin{bibdiv}
\begin{biblist}

\bib{BM2023}{article}{
      author={Bhojak, A.},
      author={Mohanty, P.},
       title={Weak type bounds for rough maximal singular integrals near {$L^1$}},
        date={2023},
        ISSN={0022-1236,1096-0783},
     journal={J. Funct. Anal.},
      volume={284},
      number={10},
       pages={Paper No. 109881, 27},
         url={https://doi.org/10.1016/j.jfa.2023.109881},
      review={\MR{4554742}},
}

\bib{BS2026}{article}{
      author={Bhojak, A.},
      author={Shrivastava, S.},
       title={Endpoint variation and jump inequalities for rough singular integrals},
        date={2026},
      eprint={arXiv:2602.21888},
         url={https://arxiv.org/abs/2602.21888},
}

\bib{BH2019}{article}{
      author={Buri\'ankov\'a, E.},
      author={Honz\'ik, P.},
       title={Rough maximal bilinear singular integrals},
        date={2019},
        ISSN={0010-0757,2038-4815},
     journal={Collect. Math.},
      volume={70},
      number={3},
       pages={431\ndash 446},
         url={https://doi.org/10.1007/s13348-019-00239-4},
      review={\MR{3990991}},
}

\bib{CZ1956}{article}{
      author={Calder\'on, A.~P.},
      author={Zygmund, A.},
       title={On singular integrals},
        date={1956},
        ISSN={0002-9327,1080-6377},
     journal={Amer. J. Math.},
      volume={78},
       pages={289\ndash 309},
         url={https://doi.org/10.2307/2372517},
      review={\MR{84633}},
}

\bib{Christ1988}{article}{
      author={Christ, M.},
       title={Weak type {$(1,1)$} bounds for rough operators},
        date={1988},
        ISSN={0003-486X,1939-8980},
     journal={Ann. of Math. (2)},
      volume={128},
      number={1},
       pages={19\ndash 42},
         url={https://doi.org/10.2307/1971461},
      review={\MR{951506}},
}

\bib{CR1988}{article}{
      author={Christ, M.},
      author={Rubio~de Francia, J.~L.},
       title={Weak type {$(1,1)$} bounds for rough operators. {II}},
        date={1988},
        ISSN={0020-9910,1432-1297},
     journal={Invent. Math.},
      volume={93},
      number={1},
       pages={225\ndash 237},
         url={https://doi.org/10.1007/BF01393693},
      review={\MR{943929}},
}

\bib{ChristZhou}{article}{
      author={Christ, M.},
      author={Zhou, Z.},
       title={A class of singular bilinear maximal functions},
        date={2024},
        ISSN={0022-1236,1096-0783},
     journal={J. Funct. Anal.},
      volume={287},
      number={8},
       pages={Paper No. 110572},
         url={https://doi.org/10.1016/j.jfa.2024.110572},
      review={\MR{4776384}},
}

\bib{CM1978}{article}{
      author={Coifman, R.},
      author={Meyer, Y.},
       title={Commutateurs d'int\'{e}grales singuli\`eres et op\'{e}rateurs multilin\'{e}aires},
        date={1978},
        ISSN={0373-0956,1777-5310},
     journal={Ann. Inst. Fourier (Grenoble)},
      volume={28},
      number={3},
       pages={xi, 177\ndash 202},
         url={http://www.numdam.org/item?id=AIF_1978__28_3_177_0},
      review={\MR{511821}},
}

\bib{CF1975}{article}{
      author={Coifman, R.~R.},
      author={Meyer, Y.},
       title={On commutators of singular integrals and bilinear singular integrals},
        date={1975},
        ISSN={0002-9947,1088-6850},
     journal={Trans. Amer. Math. Soc.},
      volume={212},
       pages={315\ndash 331},
         url={https://doi.org/10.2307/1998628},
      review={\MR{380244}},
}

\bib{DGHST2011}{article}{
      author={Diestel, G.},
      author={Grafakos, L.},
      author={Honz\'ik, P.},
      author={Si, Z.},
      author={Terwilleger, E.M.},
       title={Method of rotations for bilinear singular integrals},
        date={2011},
        ISSN={1938-9787},
     journal={Commun. Math. Anal.},
       pages={99\ndash 107},
      review={\MR{2772055}},
}

\bib{DPS2024}{article}{
      author={Dosidis, G.},
      author={Park, B.~J.},
      author={Slav\'ikov\'a, L.},
       title={Boundedness criteria for bilinear fourier multipliers via shifted square function estimates},
        date={2024},
      eprint={arXiv:2402.15785},
         url={https://arxiv.org/abs/2402.15785},
}

\bib{DosidisSlavikova}{article}{
      author={Dosidis, G.},
      author={Slav\'ikov\'a, L.},
       title={Multilinear singular integrals with homogeneous kernels near {$L^1$}},
        date={2024},
        ISSN={0025-5831,1432-1807},
     journal={Math. Ann.},
      volume={389},
      number={3},
       pages={2259\ndash 2271},
         url={https://doi.org/10.1007/s00208-023-02691-x},
      review={\MR{4753062}},
}

\bib{DR1986}{article}{
      author={Duoandikoetxea, J.},
      author={Rubio~de Francia, J.~L.},
       title={Maximal and singular integral operators via {F}ourier transform estimates},
        date={1986},
        ISSN={0020-9910,1432-1297},
     journal={Invent. Math.},
      volume={84},
      number={3},
       pages={541\ndash 561},
         url={https://doi.org/10.1007/BF01388746},
      review={\MR{837527}},
}

\bib{GrafakosclassicalFA}{book}{
      author={Grafakos, L.},
       title={Classical {F}ourier analysis},
     edition={Third},
      series={Graduate Texts in Mathematics},
   publisher={Springer, New York},
        date={2014},
      volume={249},
        ISBN={978-1-4939-1193-6; 978-1-4939-1194-3},
         url={https://doi.org/10.1007/978-1-4939-1194-3},
      review={\MR{3243734}},
}

\bib{GrafakosmodernFA}{book}{
      author={Grafakos, L.},
       title={Modern {Fourier} analysis},
    language={English},
     edition={3rd ed.},
      series={Grad. Texts Math.},
   publisher={New York, NY: Springer},
        date={2014},
      volume={250},
        ISBN={978-1-4939-1229-2; 978-1-4939-1230-8},
}

\bib{GHH2018}{article}{
      author={Grafakos, L.},
      author={He, D.},
      author={Honz\'{\i}k, P.},
       title={Rough bilinear singular integrals},
        date={2018},
        ISSN={0001-8708,1090-2082},
     journal={Adv. Math.},
      volume={326},
       pages={54\ndash 78},
         url={https://doi.org/10.1016/j.aim.2017.12.013},
      review={\MR{3758426}},
}

\bib{GHHP2023}{article}{
      author={Grafakos, L.},
      author={He, D.},
      author={Honz\'ik, P.},
      author={Park, B.~J.},
       title={Initial {$L^2\times\cdots\times L^2 $} bounds for multilinear operators},
        date={2023},
        ISSN={0002-9947,1088-6850},
     journal={Trans. Amer. Math. Soc.},
      volume={376},
      number={5},
       pages={3445\ndash 3472},
         url={https://doi.org/10.1090/tran/8877},
      review={\MR{4577337}},
}

\bib{GHHP2024}{article}{
      author={Grafakos, L.},
      author={He, D.},
      author={Honz\'ik, P.},
      author={Park, B.~J.},
       title={Multilinear rough singular integral operators},
        date={2024},
        ISSN={0024-6107,1469-7750},
     journal={J. Lond. Math. Soc. (2)},
      volume={109},
      number={2},
       pages={Paper No. e12867, 35},
         url={https://doi.org/10.1112/jlms.12867},
      review={\MR{4704158}},
}

\bib{GHHP2024ae}{article}{
      author={Grafakos, L.},
      author={He, D.},
      author={Honz\'ik, P.},
      author={Park, B.~J.},
       title={On pointwise a.e. convergence of multilinear operators},
        date={2024},
     journal={Canadian Journal of Mathematics},
      volume={76},
      number={3},
       pages={1005–1032},
}

\bib{GHS2019}{article}{
      author={Grafakos, L.},
      author={He, D.},
      author={Slav\'ikov\'a, L.},
       title={Failure of the {H}\"ormander kernel condition for multilinear {C}alder\'on-{Z}ygmund operators},
        date={2019},
        ISSN={1631-073X,1778-3569},
     journal={C. R. Math. Acad. Sci. Paris},
      volume={357},
      number={4},
       pages={382\ndash 388},
         url={https://doi.org/10.1016/j.crma.2019.04.002},
      review={\MR{3955209}},
}

\bib{GHS2020}{article}{
      author={Grafakos, L.},
      author={He, D.},
      author={Slav\'{\i}kov\'{a}, L.},
       title={{$L^2\times L^2\to L^1$} boundedness criteria},
        date={2020},
        ISSN={0025-5831,1432-1807},
     journal={Math. Ann.},
      volume={376},
      number={1-2},
       pages={431\ndash 455},
         url={https://doi.org/10.1007/s00208-018-1794-5},
      review={\MR{4055166}},
}

\bib{GL2004}{article}{
      author={Grafakos, L.},
      author={Li, X.},
       title={Uniform bounds for the bilinear {H}ilbert transforms. {I}},
        date={2004},
        ISSN={0003-486X,1939-8980},
     journal={Ann. of Math. (2)},
      volume={159},
      number={3},
       pages={889\ndash 933},
         url={https://doi.org/10.4007/annals.2004.159.889},
      review={\MR{2113017}},
}

\bib{GT2002}{article}{
      author={Grafakos, L.},
      author={Torres, R.~H.},
       title={Multilinear {C}alder\'{o}n-{Z}ygmund theory},
        date={2002},
        ISSN={0001-8708,1090-2082},
     journal={Adv. Math.},
      volume={165},
      number={1},
       pages={124\ndash 164},
         url={https://doi.org/10.1006/aima.2001.2028},
      review={\MR{1880324}},
}

\bib{HP2023}{article}{
      author={He, D.},
      author={Park, B.~J.},
       title={Improved estimates for bilinear rough singular integrals},
        date={2023},
        ISSN={0025-5831,1432-1807},
     journal={Math. Ann.},
      volume={386},
      number={3-4},
       pages={1951\ndash 1978},
         url={https://doi.org/10.1007/s00208-022-02444-2},
      review={\MR{4612410}},
}

\bib{Hofmann1988}{article}{
      author={Hofmann, S.},
       title={Weak {$(1,1)$} boundedness of singular integrals with nonsmooth kernel},
        date={1988},
        ISSN={0002-9939,1088-6826},
     journal={Proc. Amer. Math. Soc.},
      volume={103},
      number={1},
       pages={260\ndash 264},
         url={https://doi.org/10.2307/2047563},
      review={\MR{938680}},
}

\bib{Honzik2020}{article}{
      author={Honz\'ik, P.},
       title={An endpoint estimate for rough maximal singular integrals},
        date={2020},
        ISSN={1073-7928,1687-0247},
     journal={Int. Math. Res. Not. IMRN},
      number={19},
       pages={6120\ndash 6134},
         url={https://doi.org/10.1093/imrn/rny189},
      review={\MR{4165473}},
}

\bib{HLS2025}{article}{
      author={Honz\'ik, P.},
      author={Lappas, S.},
      author={Slav\'ikov\'a, L.},
       title={Bilinear singular integral operators with kernels in weighted spaces},
        date={2025},
      eprint={arXiv:2412.07014},
         url={https://arxiv.org/abs/2412.07014},
}

\bib{IPS}{article}{
      author={Iosevich, A.},
      author={Palsson, E.~A.},
      author={Sovine, S.~R.},
       title={Simplex averaging operators: quasi-{B}anach and {$L^p$}-improving bounds in lower dimensions},
        date={2022},
        ISSN={1050-6926,1559-002X},
     journal={J. Geom. Anal.},
      volume={32},
      number={3},
       pages={Paper No. 87, 16},
         url={https://doi.org/10.1007/s12220-021-00843-6},
      review={\MR{4363760}},
}

\bib{KenigStein}{article}{
      author={Kenig, C.~E.},
      author={Stein, E.~M.},
       title={Multilinear estimates and fractional integration},
        date={1999},
        ISSN={1073-2780},
     journal={Math. Res. Lett.},
      volume={6},
      number={1},
       pages={1\ndash 15},
         url={https://doi.org/10.4310/MRL.1999.v6.n1.a1},
      review={\MR{1682725}},
}

\bib{LT1997}{article}{
      author={Lacey, M.},
      author={Thiele, C.},
       title={{$L^p$} estimates on the bilinear {H}ilbert transform for {$2<p<\infty$}},
        date={1997},
        ISSN={0003-486X,1939-8980},
     journal={Ann. of Math. (2)},
      volume={146},
      number={3},
       pages={693\ndash 724},
         url={https://doi.org/10.2307/2952458},
      review={\MR{1491450}},
}

\bib{LT1999}{article}{
      author={Lacey, M.},
      author={Thiele, C.},
       title={On {C}alder\'on's conjecture},
        date={1999},
        ISSN={0003-486X,1939-8980},
     journal={Ann. of Math. (2)},
      volume={149},
      number={2},
       pages={475\ndash 496},
         url={https://doi.org/10.2307/120971},
      review={\MR{1689336}},
}

\bib{Lai2025}{article}{
      author={Lai, X.},
       title={Weak $(1,1)$ estimate for maximal truncated rough singular integral operator},
        date={2025},
      eprint={arXiv:2508.17737},
         url={https://arxiv.org/abs/2508.17737},
}

\bib{Li2006}{article}{
      author={Li, X.},
       title={Uniform bounds for the bilinear {H}ilbert transforms. {II}},
        date={2006},
        ISSN={0213-2230,2235-0616},
     journal={Rev. Mat. Iberoam.},
      volume={22},
      number={3},
       pages={1069\ndash 1126},
         url={https://doi.org/10.4171/RMI/483},
      review={\MR{2320411}},
}

\bib{Muscalu}{article}{
      author={Muscalu, C.},
       title={Calder\'{o}n commutators and the {C}auchy integral on {L}ipschitz curves revisited: {I}. {F}irst commutator and generalizations},
        date={2014},
        ISSN={0213-2230,2235-0616},
     journal={Rev. Mat. Iberoam.},
      volume={30},
      number={2},
       pages={727\ndash 750},
         url={https://doi.org/10.4171/RMI/798},
      review={\MR{3231215}},
}

\bib{Park2024}{article}{
      author={Park, B.~J.},
       title={Vector-valued estimates for shifted operators},
        date={2024},
      eprint={arXiv:2401.17785},
         url={https://arxiv.org/abs/2401.17785},
}

\bib{Park2025}{article}{
      author={Park, B.~J.},
       title={Multilinear estimates for maximal rough singular integrals},
        date={2025},
     journal={Ann. Sc. Norm. Super. Pisa Cl. Sci.},
         url={https://doi.org/10.2422/2036-2145.202409_020},
}

\bib{Seeger1996}{article}{
      author={Seeger, A.},
       title={Singular integral operators with rough convolution kernels},
        date={1996},
        ISSN={0894-0347,1088-6834},
     journal={J. Amer. Math. Soc.},
      volume={9},
      number={1},
       pages={95\ndash 105},
         url={https://doi.org/10.1090/S0894-0347-96-00185-3},
      review={\MR{1317232}},
}

\bib{SteinHarmonicAnalysisRealVariableMethodOrtho}{book}{
      author={Stein, E.~M.},
       title={Harmonic analysis: real-variable methods, orthogonality, and oscillatory integrals},
      series={Princeton Mathematical Series},
   publisher={Princeton University Press, Princeton, NJ},
        date={1993},
      volume={43},
        ISBN={0-691-03216-5},
        note={With the assistance of Timothy S. Murphy, Monographs in Harmonic Analysis, III},
      review={\MR{1232192}},
}

\end{biblist}
\end{bibdiv}
\end{document}